\documentclass[12pt]{amsart}
\usepackage{amsmath}
\usepackage{amsxtra}
\usepackage{amscd}
\usepackage{amsthm}
\usepackage{amsfonts}
\usepackage{amssymb}
\usepackage{eucal}
\usepackage{epsfig}
\usepackage{graphics}
\usepackage{graphicx}
\usepackage{accents}
\usepackage{dynkin-diagrams}
\usepackage{comment}
\usepackage{enumitem}

\usepackage{mathtools}
\usepackage{tikz-cd}
\usepackage{mathtools}
\usepackage{pgf}
\usetikzlibrary{arrows, matrix}

\usepackage{here} 
\usepackage{simplewick}
\numberwithin{equation}{section}
\newtheorem{thm}{Theorem}[section]
\newtheorem{prop}[thm]{Proposition}
\newtheorem{lem}[thm]{Lemma}
\newtheorem{rem}[thm]{Remark}
\newtheorem{cor}[thm]{Corollary}

\newtheorem{dfn}[thm]{Definition}

\newcommand{\nn}{\nonumber}

\theoremstyle{definition}
\newtheorem{example}[thm]{Example}

\newcommand{\ds}[1]{\displaystyle #1}

\newcommand{\C}{{\mathbb C}}
\newcommand{\Z}{{\mathbb Z}}
\newcommand{\Q}{{\mathbb Q}}

\newcommand{\cA}{{\mathcal A}}
\newcommand{\cB}{{\mathcal B}}
\newcommand{\B}{{\mathcal B}}
\newcommand{\E}{{\mathcal E}}
\newcommand{\F}{\mathcal F}

\newcommand{\cR}{\mathcal{R}}

\newcommand{\bs}{\boldsymbol}

\newcommand{\gl}{\mathfrak{gl}}

\newcommand{\la}{\lambda}

\newcommand{\bk}{{\mathbf{k}}}

\newcommand{\bfs}{{\mathbf{s}}}

\newcommand{\bM}{{\mathbf{M}}}
\newcommand{\bN}{{\mathbf{N}}}

\newcommand{\ssq}{\textsf{q}}

\newcommand{\id}{{\rm id}}

\newcommand{\Ker}{\mathop{\rm Ker}}

\newcommand{\sgn}{\mathop{\rm sgn}}

\newcommand{\Tr}{{\rm Tr}}

\newcommand{\on}{\operatorname}
\newcommand{\mc}{\mathcal}
\newcommand{\al}{\alpha}

\begin{document}

\begin{title}[Commutative subalgebra of $\gl_{m|n}$ shuffle algebra]
{Commutative subalgebra of a shuffle algebra 
associated with quantum toroidal $\gl_{m|n}$}
\end{title}
\author{B. Feigin, M. Jimbo, and E. Mukhin}

\address{BF: National Research University Higher School of Economics,  101000, Myasnitskaya ul. 20, Moscow,  Russia, and Landau Institute for Theoretical Physics, 142432, pr. Akademika Semenova 1a, Chernogolovka, Russia
}
\email{bfeigin@gmail.com}
\address{MJ: 
Professor Emeritus,
Rikkyo University, Toshima-ku, Tokyo 171-8501, Japan}
\email{jimbomm@rikkyo.ac.jp}
\address{EM: Department of Mathematics,
Indiana University Purdue University Indianapolis,
402 N. Blackford St., LD 270, 
Indianapolis, IN 46202, USA}
\email{emukhin@iupui.edu}


\begin{abstract} 
We define and study the shuffle algebra $Sh_{m|n}$ of the quantum toroidal algebra $\mc E_{m|n}$ associated to Lie superalgebra $\gl_{m|n}$. We show that $Sh_{m|n}$ contains a family  of commutative subalgebras $\mc B_{m|n}(s)$ depending on parameters $s=(s_1,\dots,s_{m+n})$, $\prod_i s_i=1$, given by appropriate regularity conditions. We show that $\mc B_{m|n}(s)$ is a free polynomial algebra and give explicit generators which conjecturally correspond to the traces of the $s$-weighted $R$-matrix computed on the degree zero part of $\mc E_{m|n}$ modules of levels $\pm 1$.

\end{abstract}


\maketitle

\section{Introduction}
We continue the study of integrable systems of XXZ-type related to quantum toroidal algebras, see \cite{FHHSY}, \cite{FT},
\cite{FJMM2}, \cite{FJMM3}, \cite{FJM1}-\cite{FJM4}, \cite{FJMV}. 
In this paper we make the first steps in the case of quantum toroidal algebra $\mc E_{m|n}$ of type $\gl_{m|n}$.

The quantum toroidal algebra $\mc E_{m|n}$ of type $\gl_{m|n}$ was introduced and studied in \cite{BM1}, \cite{BM2}. The algebra $\mc E_{m|n}$ depends on parameters $q_1,q_2,q_3$, $q_1q_2q_3=1$, and has several presentations depending on the choice of the Dynkin diagram for $\gl_{m|n}$. In this paper we use only the standard parity. It is expected that $\mc E_{m|n}$
has a universal  $R$-matrix 
and therefore one could construct a quantum integrable system generated by the transfer matrices. 

In the even case, (that is in the case of $n=0$), the corresponding integrals of motion can be written explicitly  in the form of multidimensional contour integrals, \cite{FJM3}. These formulas can be used to find interesting generalizations, see \cite{FJMV}, \cite{FJM4}, to prove dualities, \cite{FJM2}, to look for connections to other models, \cite{FJMV}.
 The spectrum of integrals of motion is known to be given in terms of the Bethe ansatz, \cite{FJMM1}, \cite{FJMM2}, \cite{AO}. 
All of that is missing so far in the supersymmetric case.

In this paper we take the shuffle algebra approach developed in the even case  in \cite{FT}.  The shuffle algebra is a remarkable realization of the nilpotent part of a quantum group by spaces of rational functions symmetric in several groups of variables with explicit poles and satisfying explicit vanishing wheel conditions, see \cite{N1}, \cite{N2},
\cite{FT}, \cite{T1}, \cite{T2}, 
Sections \ref{ap1 sec} and \ref{sec:shflm}. The product of such  functions in the shuffle algebra is given by the shuffle star product, see \eqref{shfl}. The exchange functions in the star product come from the quadratic relations, and the wheel conditions from the Serre relations.

We introduce the shuffle algebra $Sh_{m|n}$ associated to $\mc E_{m|n}$. The definition $Sh_{m|n}$ is rather standard. However, the isomorphism of the shuffle algebra and the nilpotent part of $\mc E_{m|n}$ is not trivial, and we do not pursue it here, see \cite{N1}, \cite{N2}  for the even case.  Instead, we find a family of commutative subalgebras $\mc B_{m|n}(s)$ of $\mc E_{m|n}$ depending on parameters $s=(s_1,\dots,s_{m+n})$, $s_1s_2\cdots s_{m+n}=1$,
which we call the {\it Bethe algebras}. 
We use two main ideas. First, it is expected that the Bethe algebra
is characterized by 
the ratio of two limits when some of the variables in the rational function $F\in Sh_{m|n}$ go to zero and when the same variables go to infinity, see \eqref{ratio of limits}. Second, we can compute a few of elements of the Bethe algebra explicitly using the trace of the hypothetical $R$-matrix, see the Appendix. This happens to be sufficient to conjecture the answer, and then we prove it.

Our arguments often follow the even case of \cite{FT}.  However, there are several technical differences worth mentioning. 

First, in the bosonic realization of the $\mc E_{m|n}$ level one modules, part of the generators $E_i(z)$ act not as a single vertex operator but as a sum of two, see \cite{BM1}. As the result, in contrast to the even case, the matrix coefficients of the $R$-matrix do not factorize.  Instead they are given in terms of products of functions $I^c_{M,N}$ of two groups of variables:
\begin{align*}
 I^c_{M,N}(y_1,\dots,y_M;z_1,\dots,z_N)=   \frac{1}{\prod_{i<j}(y_i-y_j)}
\prod_{a=1}^M\Bigl(\frac{q^cT_{q,y_a}-q^{-c}T^{-1}_{q,y_a}}{q-q^{-1}}\Bigr)
\Bigl(\frac{\prod_{i<j}(y_i-y_j)}{\prod_{i,j}(y_i-z_j)}
\Bigr)\,,
\end{align*}
where 
$(T_{q,y_a}f)(y_1,\ldots,y_M)=f(y_1,\ldots,q y_a,\ldots,y_M).$  The functions $I^c_{M,N}$ satisfy a number of identities and curious properties, see 
Lemmas \ref{lem:Ic-prop}--\ref{lem:Iwheel},
\eqref{identity}, and seem to be an important ingredient.

Second, in the proof we make use of the fusion homomorphism which maps the subalgebra of $Sh_{m|n}$ of weight zero to a subalgebra of weight zero in $Sh_{m|n-1}\otimes Sh_{1|0}$ (with different parameters $ q_1, q_2$), see Proposition \ref{prop:fusion-glmn}, cf. \cite{FJM1}. It allows us to use the induction on the rank of the algebra. This homomorphism is not present in \cite{FT}.

Third, unlike \cite{FT}, we are not able to show that our commutative subalgebra coincides with the algebra of transfer matrices restricted to degree zero. We do expect this to be true, see the Appendix.

Fourth, similar to \cite{FT}, we show that the Bethe algebra $\mc B_{m|n}$ is a free polynomial algebra with explicit generators which we expect to correspond to the first $m$ skew-symmetric powers of the vector representations and first $n$ skew-symmetric powers of the covector representations of $\gl_{m|n}$.

Finally, it is
desirable to give the elliptic deformation of $\mc B_{m|n}$ corresponding to the full trace of the $R$-matrix (not restricted to the degree zero) as in \cite{FT}  for the even case and translate this knowledge into explicit integrals of motion as in \cite{FJM3} in the even case.  

\medskip

Note also, that the case $m=n$ is excluded from our considerations. While many definitions make sense in that case too, the proofs seem to fall apart. This case was not treated in \cite{BM1}, \cite{BM2} either.

\medskip

The paper is constructed as follows. In Section \ref{sec:Shfl} we describe the generalities about the $\gl_{m|n}$ shuffle algebra. In Section \ref{fusion sec} we give the fusion homomorphism of shuffle algebras which is used later for the induction with respect to the rank. The main results of this section are 
Proposition \ref{prop:fusion-glmn} and Proposition \ref{prop:fusion_glm}. 
In Section \ref{Bethe sec} we define the Bethe subalgebra of the shuffle algebra. We prove an upper estimate for the size of the Bethe algebra, Proposition \ref{prop:Gordon}. In Section \ref{G functions sec}, we give explicit generators of the Bethe algebra. We finally prove that the Bethe algebra is commutative in Theorem \ref{prop:surj_glmn}. In the Appendix,  we give examples of computation with traces of $R$-matrices which illustrates the origin of our formulas.

\bigskip

\noindent{\it Notation.}\quad
Throughout the paper we fix $q,d\in\C^{\times}$ and set
\begin{align*}
q_1=q^{-1}d\,,\quad q_2=q^2\,,\quad q_3=q^{-1}d^{-1}\,.
\end{align*}
We have $q_1q_2q_3=1$. 
We assume that $q_1^{n_1}q_2^{n_2}q_3^{n_3}=1$, $n_1,n_2,n_3\in\Z$, holds
only if $n_1=n_2=n_3$. 
Fixing a branch of logarithm $\log q_i$ we set $q_i^r=e^{r\log q_i}$ for 
$r\in\Q$. 

If a statement $P$ is true we set $\theta(P)=1$, 
and $\theta(P)=0$ otherwise.

\section{Shuffle algebras}\label{sec:Shfl}

\subsection{Algebra $Sh_{m|n}$}\label{sec:shflmn}

Let $m,n$ be positive integers with $m\neq n$. 
In this section we define
the shuffle algebra $Sh_{m|n}$ for  the quantum toroidal algebra 
of type $\gl_{m|n}$. 

Let $Sh_{m|n}$ be a graded complex vector space:
\begin{align*}
&Sh_{m|n} =\bigoplus_{\bN \in\Z_{\ge0}^{m+n}}
Sh_{m|n,\bN}\,,
\end{align*}
where
each graded component $Sh_{m|n,\bN}$ 
of degree $\bN=(N_1,\ldots,N_{m+n})$
is the space of all rational functions $F(x_1;\cdots;x_{m+n})$ 
in the variables $x_i=(x_{i,1},\ldots,x_{i,N_i})$, $1\le i\le m+n$, 
satisfying the following conditions (i)--(iii) below. 

Here and after the index
$i$ of $x_i$ is to be read modulo $m+n$, i.e.
 $x_{i+m+n,a}=x_{i,a}$. 
We call this index the color of the variable. We call variables $x_{m,a}, x_{m+n,a}$ fermionic, all other variables bosonic. We say that a bosonic variabel $x_{i,a}$ is before the equator if $i=1,\dots,m-1$ and after the equator if $i=m+1,\dots,m+n-1$.

\begin{enumerate}
 \item $F$
 has the form 
\begin{align*}
F(x_1;\cdots;x_{m+n}) 
=\frac{f(x_1;\cdots;x_{m+n}) }{\prod_{i=1}^{m+n}
\prod_{1\le a \le N_i\atop 1\le b\le N_{i+1}}(x_{i,a}-x_{i+1,b})}\,,
\end{align*}
where $f(x_1;\cdots;x_{m+n})$ is a Laurent polynomial. 
\item $F$ is symmetric in bosonic variables $x_i$, $i\neq m,m+n$, and  
skew-symmetric in fermionic variables $x_{m}$, $x_{m+n}$.
\item $F$ vanishes if one of the following holds for some $a,b,c,d$:
\begin{enumerate}
 \item $x_{i,b}=q_3x_{i-1,a}$ and $x_{i,c}=q_2x_{i,b}$, 
$1\le i\le m-1$;
 \item 
$x_{i,b}=q_2x_{i,a}$ and $x_{i+1,c}=q_3x_{i,b}$, 
$1\le i\le m-1$;
 \item 
$x_{m,b}=q_3 x_{m-1,a}$, $x_{m+1,c}=q_3^{-1}x_{m,b}$, 
and $x_{m,d}=q_1^{-1}x_{m+1,c}$;
\item  
$x_{i,b}=q_3^{-1}x_{i-1,a}$ and $x_{i,c}=q_2^{-1}x_{i,b}$, 
$m+1\le i\le m+n-1$;
 \item 
$x_{i,b}=q_2^{-1}x_{i,a}$ and $x_{i+1,c}=q_3^{-1}x_{i,b}$, 
$m+1\le i\le m+n-1$;
 \item 
$x_{m+n,b}=q_3^{-1} x_{m+n-1,a}$, $x_{1,c}=q_3x_{m+n,b}$, 
and $x_{m+n,d}=q_1 x_{1,c}$.
\end{enumerate}
\end{enumerate}
We call (iii) {\it wheel conditions}. 
The wheel conditions are schematically shown in
Figure \ref{fig:wheel} below. In that figure, for simplicity, we omitted the second indices of variables.

Note that all cubic wheel conditions contain two bosonic variables of the same color. The quartic relations contain two fermionic variables of the same color.

Note also that, due to skew-symmetry, $F$ has obvious zeroes at 
\begin{align}
x_{m,a}=x_{m,b}\,\quad\text{or}\quad  x_{m+n,a}=x_{m+n,b} \,,
\quad a\neq b\,.
\label{trivialzero}
\end{align}

\bigskip

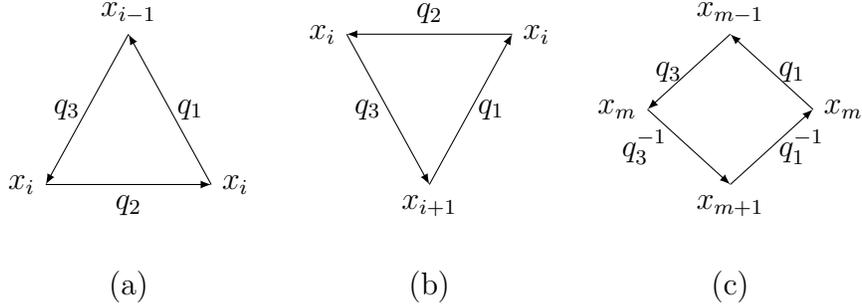
\begin{figure}[H]
\begin{align*}
\begin{tikzpicture}
\draw[-latex](-4,1)--(-5.1,-1);
\draw[-latex](-2.9,-1)--(-4,1);
\draw[-latex](-5.1,-1)--(-2.9,-1);
\draw(-4,1)node[above]{$x_{i-1}$};
\draw(-5.1,-1)node[left]{$x_i$};
\draw(-2.9,-1)node[right]{$x_i$};
\draw(-4.5,0)node[left]{$q_3$};
\draw(-4,-1)node[below]{$q_2$};
\draw(-3.5,0)node[right]{$q_1$};
\draw(-4,-2)node[below]{(a)};
\draw[-latex](-1.1,1)--(0,-1);
\draw[-latex](0,-1)--(1.1,1);
\draw[-latex](1.1,1)--(-1.1,1);
\draw(0,-1)node[below]{$x_{i+1}$};
\draw(-1.1,1)node[left]{$x_i$};
\draw(1.1,1)node[right]{$x_i$};
\draw(-0.5,0)node[left]{$q_3$};
\draw(0,1)node[above]{$q_2$};
\draw(0.5,0)node[right]{$q_1$};
\draw(0,-2)node[below]{(b)};
\draw[-latex](4,1)--(2.9,0);
\draw[-latex](2.9,0)--(4,-1);
\draw[-latex](4,-1)--(5.1,0);
\draw[-latex](5.1,0)--(4,1);
\draw(4,1)node[above]{$x_{m-1}$};
\draw(2.9,0)node[left]{$x_m$};
\draw(5.1,0)node[right]{$x_m$};
\draw(4,-1)node[below]{$x_{m+1}$};
\draw(3.5,0.5)node[left]{$q_3$};
\draw(3.3,-0.5)node[left]{$q_3^{-1}$};
\draw(4.5,0.5)node[right]{$q_1$};
\draw(4.5,-0.5)node[right]{$q_1^{-1}$};
\draw(4,-2)node[below]{(c)};
\end{tikzpicture}
\end{align*}
\caption{Wheel conditions (a), (b), (c). 
Here $i$ is bosonic before the equator.
Conditions (d), (e), (f) are obtained by reversing
all arrows. In these conditions $m$ is replaced with $m+n$ and $i$ is bosonic after the equator.
}\label{fig:wheel}
\end{figure}

Define {\it structure functions} $\omega_{i,j}(x,y)$ by
\begin{align*}
&\omega_{i,i}(x,y)=
\begin{cases}
\ds{\frac{x-q_2y}{x-y}}& \text{if $1\le i\le m-1$},\\		     
\ds{\frac{x-q_2^{-1}y}{x-y}}& \text{if $m+1\le i\le m+n-1$},\\		   
\end{cases} 
\\
&\omega_{i-1,i}(x,y)=
\begin{cases}
\ds{\frac{x-q_1y}{x-y}}& \text{if $1\le i\le m$},\\		     
\ds{\frac{x-q_1^{-1}y}{x-y}}& \text{if $m+1\le i\le m+n$},\\		     
\end{cases} 
\\
&\omega_{i,i-1}(x,y)=
\begin{cases}
\ds{d\frac{x-q_3y}{x-y}}& \text{if $1\le i\le m$},\\
\ds{d^{-1}
\frac{x-q_3^{-1}y}{x-y}}
& \text{if $m+1\le i\le m+n$}.\\		     
\end{cases} 
\end{align*}
We set $\omega_{i,j}(x,y)=1$ in all other cases.

For $F\in Sh_{m|n,\bM}$ and $G\in Sh_{m|n,\bN}$, we 
define their {\it shuffle product} $F*G$
by setting 
\begin{align}
&(F*G)(x_1;\cdots;x_{m+n})=\sum_{I_1,J_1}\cdots\sum_{I_{m+n},J_{m+n}} 
\sgn(I_m,J_m)\sgn(I_{m+n},J_{m+n})
\label{shfl}
\\
&\times F(x_{1,I_1};\cdots;x_{m+n,I_{m+n}})
 G(x_{1,J_1};\cdots;x_{m+n,J_{m+n}})
\prod_{i,j=1}^{m+n}\prod_{a\in I_i, b\in J_j}\omega_{i,j}(x_{i,a},x_{j,b})
\,.\nn
\end{align}
Here the sum is taken over all subsets $I_i,J_i$ of $\{1,\ldots,M_i+N_i\}$ 
such that
\begin{align*}
 I_i\sqcup J_i=\{1,\ldots,M_i+N_i\}\,,
\quad |I_i|=M_i\,,\ |J_i|=N_i\,,
\quad i=1,\ldots,m+n\,.
\end{align*}
For a subset 
$I=\{a_1,\ldots,a_k\}$, $a_1<\cdots<a_k$,
we write $x_{i,I}=(x_{i,a_1},\ldots,x_{i,a_k})$.
If $I=\{a_1,\ldots,a_{k}\}$, $a_1<\cdots<a_k$,
 and $J=\{b_1,\ldots,b_{l}\}$, $b_1<\cdots<b_l$, then $\sgn(I,J)$ stands for the signature of the permutation
$(a_1,\ldots,a_{k},b_1,\ldots,b_{l})$.

Then we have the following standard lemma.
\begin{lem}
 The space $Sh_{m|n}$ with the shuffle product is    an associative $\Z_{\ge0}^{m+n}$-graded algebra.
\end{lem}
\begin{proof}
Associativity is obvious from the definition. It is straightforward to check that if $F\in Sh_{m|n,\bM}$ and $G\in Sh_{m|n,\bN}$,  
then the shuffle product $F*G$ also satisfies the wheel conditions and therefore  
$F*G\in Sh_{m|n,\bM+\bN}$.
\end{proof}

As necessary, the dependence of shuffle algebras on parameters 
will be exhibited explicitly as $Sh_{m|n}(q_1,q_2,q_3)$.

We have an isomorphism
\begin{align}
&Sh_{m|n}(q_1,q_2,q_3)  \longrightarrow Sh_{n|m}(q^{-1}_1,q^{-1}_2,q^{-1}_3),
\label{cyclmn}\\
&F(x_1;\cdots;x_{m+n})\mapsto F(x_{m+1};\cdots; x_{m+n};x_1;\cdots;x_m)\,.
\nn
\end{align}

We expect that as algebra $Sh_{m|n}$ is generated by 
$\bigoplus_{i=1}^{m+n}Sh_{m|n,\bs 1_i}$,
where $\bs 1_i=(0,\dots,0,1,0,\dots,0)$ with $1$ in the $i$-th position. 

The quantum toroidal algebras $\mc E_{m|n}$ associated to $\gl_{m|n}$ were defined and studied in \cite{BM1}, \cite{BM2}. 
These algebras have generators $E_{i,k}, F_{i,k}$, $K_{i,0}^{\pm1}$, $H_{i,j}$,
 $C$, $i=1,\dots, m+n$, $k\in\Z, j\in\Z\setminus\{0\}$. 

We expect that the shuffle algebra $Sh_{m|n}$ is isomorphic to the subalgebra of $\mc E_{m|n}$ generated by $F_{i,k}$. 
Under this isomorphism $F_{i,k}$  
is related to $x_{i,1}^k\in Sh_{m|n,\bs 1_i}$.

In this text we work with the standard parity, where we have only two fermionic colors $m$ and $m+n$. The shuffle algebras in other parities are defined similarly.

\subsection{Shuffle algebra $Sh_{m}$}\label{sec:shflm}
We shall use also the shuffle algebra $Sh_m$ for the quantum toroidal algebra 
of type $\gl_m$. We recall the definition below.
Further information can be found in 
the literature \cite{N1}, \cite{N2}, \cite{FT}.

Assume $m\ge3$.

Let $Sh_m$ be the graded complex vector space
\begin{align*}
Sh_m=\bigoplus_{\bN\in \Z_{\ge0}^m} Sh_{m,\bN}\,,
\end{align*}
where $Sh_{m,\bN}$ is 
the space of all rational functions 
$F(x_1;\cdots;x_m)$, $x_i=(x_{i,1},\ldots,x_{i,N_i})$, such that 
\begin{enumerate}
 \item $F$  has the form 
\begin{align}
F(x_1;\cdots;x_{m}) 
=\frac{f(x_1;\cdots;x_{m}) }{\prod_{i=1}^{m}
\prod_{1\le a \le N_i\atop 1\le b\le N_{i+1}}(x_{i,a}-x_{i+1,b})}\,,
\label{glm_pole}
\end{align}
where $f(x_1;\cdots;x_{m})$ is a Laurent polynomial; 
\item $F$ is symmetric in $x_i$ for all $1\le i\le m$;
\item $F$ vanishes if one of the following holds for some $a,b,c$:
\begin{enumerate}
 \item $x_{i,b}=q_3x_{i-1,a}$ and $x_{i,c}=q_2x_{i,b}$, $1\le i\le m$;
 \item $x_{i,b}=q_2x_{i,a}$ and $x_{i+1,c}=q_3x_{i,b}$, 
$0\le i\le m-1$.
\end{enumerate}
\end{enumerate}
Here $x_{i+m,a}=x_{i,a}$. 
The shuffle product is defined by the same formula \eqref{shfl}, 
wherein the $\sgn$ factors are dropped and the structure functions
$\omega_{i,j}(x,y)$ are replaced by
\begin{align*}
&\omega_{i,i}(x,y)=\frac{x-q_2y}{x-y}\,,
\\
&\omega_{i-1,i}(x,y)=\frac{x-q_1y}{x-y}\,,
\\
&\omega_{i,i-1}(x,y)=d
\frac{x-q_3y}{x-y}\,,
\end{align*}
and $\omega_{i,j}(x,y)=1$ in all other cases.
\medskip

In this case all variables are bosonic. The wheel conditions correspond to triangles in Figure \ref{fig:wheel}. 

\medskip

When $m=1,2$ the definition of $Sh_m$ needs a slight modification. 

In both cases the wheel conditions (a),(b) apply under the 
convention  $x_{i+m,a}=x_{i,a}$ ($m=1,2$). 

For $m=2$, the structure functions are changed to
\begin{align*}
&\omega_{i,i}(x,y)=\frac{x-q_2y}{x-y}\,,\\
&\omega_{i,i-1}(x,y)=\frac{(x-q_1y)(x-q_3y)}{(x-y)^2} \,.
\end{align*}
For $m=1$, the pole condition \eqref{glm_pole} reads
\begin{align*}
F(x)=\frac{f(x)}{\prod_{1\le a<b\le N}(x_a-x_b)^2},
\end{align*}
where $f(x)$ is a symmetric Laurent polynomial
in $x=(x_1,\ldots,x_N)$. 
The structure functions are modified to
\begin{align*}
\omega_{1,1}(x,y)=\frac{(x-q_1y)(x-q_2y)(x-q_3y)}{(x-y)^3}\,. 
\end{align*}

In what follows we write $Sh_{m}=Sh_{m|0}$. 

\section{Fusion homomorphism}\label{fusion sec}

Consider the subalgebra of $Sh_{m|n}$ consisting of subspaces of 
degree $N_1=\cdots=N_{m+n}$:
\begin{align*}
Sh^0_{m|n}=\bigoplus_{N\ge0} Sh^0_{m|n,N}\,,
\quad 
Sh^0_{m|n,N}=Sh_{m|n,N,\ldots,N}\,. 
\end{align*}

The aim of this section is to introduce homomorphisms of algebras
of the form 
\begin{align*}
&Sh^0_{m|n}(q_1,q_2,q_3)
\longrightarrow Sh^0_{m|n-1}(\bar q_1,q_2,\bar q_3)
\otimes Sh_1(q_1^{-1}q_3^{-m+n-1},q_2^{-1},q_3^{m-n})\,,
\\
&Sh^0_{m}(q_1,q_2,q_3)
\longrightarrow Sh^0_{m-1}(\tilde{q}_1,q_2,\tilde{q}_3)
\otimes Sh_1(q_1q_3^{-m+1},q_2,q_3^{m})
\end{align*}
with appropriate parameters $\bar{q}_s, \tilde{q}_s$.
We call them {\it fusion homomorphisms}. 

We are motivated by the constructions of \cite{FJMM1}. In that work, fused currents were used to construct subalgebras $\mathcal E_{m-1}\otimes \mathcal E_1\subset\mathcal E_m$ which is the quantum toroidal version of the Lie 
subalgebra inclusion $\gl_{m-1}\oplus \gl_1\subset \gl_m$. 

\subsection{Fusion homomorphism for $Sh_{m|n}$}\label{sec:fusion-mn}
In this subsection we assume that $m\ge n+1$, $n\ge 1$. 
We set 
\begin{align}
&\bar{q}_1=q_1 q_1^{-\frac{1}{m-n+1}}\,,
\quad \bar{q}_3=q_3 q_1^{\frac{1}{m-n+1}}\,.
\label{qbar}
\end{align}

Fix $L,N$ with $0\le L\le N$. 
We introduce new variables $y_i=(y_{i,a})_{1\le a\le L}$, $1\le i\le m+n-1$, 
$z=(z_b)_{L+1\le b\le N}$, 
and consider the following  
specialization of the variables $x_i=(x_{i,1},\ldots,x_{i,N})$, see  Figure \ref{fig:spv} below:
\begin{align}
&\text{$1\le a \le L$}:&&x_{i,a}=
\begin{cases}
q_1\ssq^i
y_{i,a}&\text{if $1\le i\le m$},\\ 
q_1\ssq^{2m-i}
y_{i,a}&\text{if $m+1\le i\le m+n-1$},\\ 
q_1y_{m+n-1,a}&\text{if $i=m+n$},\\ 
\end{cases} 
\label{xtoy}\\
&\text{$L+1\le b \le N$}:&&x_{i,b}=
\begin{cases}
q_3^{i-1}z_b & \text{if $0\le i\le m-1$}, \\
q_3^{2m-i-1}z_b & \text{if $m\le i\le m+n-1$}. \\
\end{cases} 
\label{xtoz}
\end{align}
where $\ssq=q_1^{-1}\bar{q}_1=q_1^{-1/(m-n+1)}$.
We have in particular
\begin{align*}
&\text{$1\le a \le L$}:
&&x_{1,a}=\bar{q}_1y_{1,a},\quad
x_{m+n-1,a}=y_{m+n-1,a},\quad x_{m+n,a}=q_1y_{m+n-1,a},\\
&\text{$L+1\le b \le N$}:&
&x_{1,b}=z_b\,,\quad x_{m+n-1,b}=q_3^{m-n}z_{b}\,,\quad
 x_{m+n,b}=q_3^{-1}z_b\,. 
\end{align*}

\bigskip

\begin{figure}[H]
\begin{align*}
\begin{tikzpicture}
\coordinate(A0) at (0.5,-0.86);
\coordinate(A1) at (0.86,-0.5);
\coordinate(A2) at (1,0);
\coordinate(A3) at (0.86,0.5);
\coordinate(A4) at (0.5,0.86);
\coordinate(A5) at (0,1);
\coordinate(A6) at (-0.5,0.86);
\coordinate(A7) at (-0.86,0.5);
\coordinate(A8) at (-0,-1);
\draw(A0) arc(-60:-30:1cm);
\draw(A1) arc(-30:0:1cm);
\draw(A2) arc(0:30:1cm);
\draw(A3) arc(30:60:1cm);
\draw[very thick,red](A4) arc(60:90:1cm);
\draw(A5) arc(90:120:1cm);
\draw(A6) arc(120:150:1cm);
\draw[dashed](A7) arc(150:270:1cm);
\draw(A8) arc(270:300:1cm);
\filldraw[white](A0) circle(0.12cm);
\filldraw[white](A1) circle(0.12cm);
\filldraw[white](A2)circle(0.12cm);
\filldraw[white](A3)circle(0.12cm);
\filldraw[white](A4)circle(0.12cm);
\filldraw[white](A5)circle(0.12cm);
\filldraw[white](A6)circle(0.12cm);
\draw(A0) circle(0.12cm);
\draw(A1) node{$\otimes$};
\draw[dashed](0.96,-0.26) arc(-15:15:1cm);
\draw(A3) circle(0.12cm);
\draw(A4) circle(0.12cm);
\draw(A5) node{$\otimes$};
\draw(A6) circle(0.12cm);
\draw(A0) node[right]{\tiny $m-1$};
\draw(A1) node[right]{\tiny $m$};
\draw(A4) node[right]{\tiny $m+n-1$};
\draw(A5) node[above]{\tiny $m+n$};
\draw(A6) node[left]{\tiny $1$};
\coordinate(B0) at (5.5,-0.86);
\coordinate(B1) at (5.86,-0.5);
\coordinate(B2) at (6,0);
\coordinate(B3) at (5+0.86,0.5);
\coordinate(B4) at (5+0.5,0.86);
\coordinate(B5) at (5,1);
\coordinate(B6) at (5-0.5,0.86);
\coordinate(B7) at (5-0.86,0.5);
\coordinate(B8) at (5-0,-1);
\draw(B0)[very thick,blue] arc(-60:-30:1cm);
\draw(B1)[very thick,green] arc(-30:0:1cm);
\draw(B2)[very thick,green] arc(0:30:1cm);
\draw(B3)[very thick,green] arc(30:60:1cm);
\draw(B4) arc(60:90:1cm);
\draw(B5)[very thick,blue] arc(90:120:1cm);
\draw(B6)[very thick,blue] arc(120:150:1cm);
\draw[dashed,very thick,blue](B7) arc(150:270:1cm);
\draw(B8)[very thick,blue] arc(270:300:1cm);
\filldraw[white](B0) circle(0.12cm);
\filldraw[white](B1) circle(0.12cm);
\filldraw[white](B2)circle(0.12cm);
\filldraw[white](B3)circle(0.12cm);
\filldraw[white](B4)circle(0.12cm);
\filldraw[white](B5)circle(0.12cm);
\filldraw[white](B6)circle(0.12cm);
\draw(B0) circle(0.12cm);
\draw(B1) node{$\otimes$};
\draw[dashed,very thick,green](5.96,-0.26) arc(-15:15:1cm);
\draw(B3) circle(0.12cm);
\draw(B4) circle(0.12cm);
\draw(B5) node{$\otimes$};
\draw(B6) circle(0.12cm);
\draw(B0) node[right]{\tiny $m-1$};
\draw(B1) node[right]{\tiny $m$};
\draw(B4) node[right]{\tiny $m+n-1$};
\draw(B5) node[above]{\tiny $m+n$};
\draw(B6) node[left]{\tiny $1$};
\end{tikzpicture}
\end{align*}
\caption{Specialization \eqref{xtoy}, \eqref{xtoz}.}\label{fig:spv}
\end{figure}
\bigskip

For a function 
$F(x_1;\cdots;x_{m+n})$, let
$\rho_{N,L}(F)(y_1;\cdots;y_{m+n-1};z)$ 
denote the function obtained by the substitution 
\eqref{xtoy}--\eqref{xtoz}.

\begin{lem}
For an element 
\begin{align*}
F(x_1;\cdots;x_{m+n}) 
=\frac{f(x_1;\cdots;x_{m+n}) }{\prod_{i=1}^{m+n}\prod_{1\le a, b\le N}(x_{i,a}-x_{i+1,b})}\,
\in Sh^0_{m|n,N}\,,
\end{align*}
define 
\begin{align*}
&\bar{F}_{N,L}(y_1;\cdots;y_{m+n-1}|z) 
\\
&=\frac{\bar{f}_{N,L}(y_1;\cdots;y_{m+n-1}|z)}
{\prod_{i=1}^{m+n-2}\prod_{1\le a, b\le N}(y_{i,a}-y_{i+1,b})
\cdot \prod_{1\le a, b\le N}(y_{m+n-1,a}-y_{1,b})
\cdot \prod_{1\le a<b\le N}(z_a-z_b)^2}
\end{align*}
as follows. 

If $n>1$ we define $\bar{f}_{N,L}$ by
\begin{align*}
&\rho_{N,L}\bigl(f(x_1;\cdots;x_{m+n})\bigr)
=const.
\bar{f}_{N,L}(y_1;\cdots;y_{m+n-1}|z)
\\
&\times 
\prod_{1\le a,b\le L}(y_{m+n-1,a}-q_2 y_{m+n-1,b}) \cdot
\prod_{L+1\le a,b\le N}(z_a-q_2z_b)^{m+n-2}\prod_{L+1\le a<b\le N}(z_a-z_b)^2
\\
&\times 
\prod_{i=1}^{m-1}\prod_{1\le a\le L\atop L+1\le b\le N}
(q_1\ssq^i y_{i,a}-q_2q_3^{i-1}z_b)(q_1\ssq^i y_{i,a}-q_2^{-1}q_3^{i-1}z_b)
\\
&\times
\prod_{i=m+1}^{m+n-1}\prod_{1\le a\le L\atop L+1\le b\le N}
(q_1\ssq^{2m-i}y_{i,a}-q_2q_3^{2m-i-1}z_b)(q_1\ssq^{2m-i}y_{i,a}-q_2^{-1}q_3^{2m-i-1}z_b)
\\
&\times\prod_{1\le a\le L\atop L+1\le b\le N}
(q_1\ssq^m y_{m,a}-q_2q_3^{m-1}z_b)(q_1\ssq^m y_{m,a}-q_3^{m-1}z_b)
(y_{m+n-1,a}-z_b)(y_{m+n-1,a}-q_2z_b)\,.
\end{align*}
If $n=1$, we replace the factor 
$\prod_{1\le a,b\le L}(y_{m+n-1,a}-q_2 y_{m+n-1,b})$
in the second line of the right hand side 
by $\prod_{a=1}^L y_{m,a}\prod_{1\le a<b\le L}(y_{m,a}-y_{m,b})^2$. 

Then $\bar{F}_{N,L}(y_1;\cdots;y_{m+n-1}|z)$ belongs to 
\begin{align*}
Sh^0_{m|n-1,L}(\bar q_1,q_2,\bar q_3)
\otimes 
Sh_{1,N-L}(q_1^{-1}q_3^{-m+n-1},q_2^{-1},q_3^{m-n})\,.
\end{align*}
\end{lem}
\begin{proof}
 The zero conditions (a)--(f) along with \eqref{trivialzero} imply that 
$\bar{f}_{N,L}(y_1;\cdots;y_{m+n-1}|z)$ defined as above 
is a Laurent polynomial. It is straightforward to check that 
the latter has the (skew-)symmetry in $y_i$ or $z$,
and that it satisfies the wheel conditions 
for $Sh_{m|n-1}(\bar{q}_1,q_2,\bar{q}_3)$ in $(y_1,\ldots,y_{m+n-1})$
and for $Sh_1(q_1^{-1}q_3^{-m+n-1},q_2^{-1},q_3^{m-n})$ in $z$.
\end{proof}
\medskip

In the definition of $\bar{f}_{N,L}$ the overall 
constant multiple is at our disposal.
We choose it 
so that in terms of $\bar{F}_{N,L}$ the formula reads as follows.
\begin{align*}
&\rho_{N,L}\bigl(F(x_1;\cdots;x_{m+n})\bigr)
=
\bar{F}_{N,L}(y_1;\cdots;y_{m+n-1}|z) \times A B C\,,
\end{align*}
where for $n>1$
\begin{align*}
A&= 
\prod_{i=0}^{m-1}\prod_{1\le a,b\le L}
\frac{y_{i,a}-y_{i+1,b}}{y_{i,a}-\ssq y_{i+1,b}}
\prod_{i=m}^{m+n-2}\prod_{1\le a,b\le L}
\frac{y_{i,a}-y_{i+1,b}}{y_{i,a}-\ssq^{-1}y_{i+1,b}}
\\
&\times 
\prod_{1\le a<b\le L}
d^{-1}
\frac{(y_{m+n-1,a}-q_2y_{m+n-1,b})(y_{m+n-1,a}-q_2^{-1}y_{m+n-1,b})}
{(y_{m+n-1,a}-q_1y_{m+n-1,b})(y_{m+n-1,a}-q_1^{-1}y_{m+n-1,b})}\,,
\\
B&=\prod_{L<a<b\le N}
d^{m-n}
\left(\frac{z_a-q_2 z_b}{z_a-q_3z_b}\frac{z_a-q_2^{-1} z_b}{z_a-q_3^{-1}z_b}
\right)^{m+n-2}
\frac{(z_a-z_b)^2}{(z_a-q_3z_b)(z_a-q_3^{-1}z_b)}
\frac{(z_a-z_b)^2}{(z_a-q_3^{m-n+1}z_b)(z_a-q_3^{-m+n-1}z_b)}\,,
\\
C&=\prod_{i=1}^{m-1}\prod_{1\le a\le L\atop L<b\le N}
d
\frac{(q_1\ssq^i y_{i,a}-q_2q_3^{i-1}z_b)(q_1\ssq^i y_{i,a}-q_2^{-1}q_3^{i-1}z_b)}
{(q_1\ssq^i y_{i,a}-q_3^{i}z_b)(q_1\ssq^i y_{i,a}-q_3^{i-2}z_b)}
\\
\times&
\prod_{1\le a\le L\atop L<b\le N}
d^{-1}
\frac{(q_1\ssq^m y_{m,a}-q_3^{m-1}z_b)(q_1\ssq^m y_{m,a}-q_2q_3^{m-1}z_b)}
{(q_1\ssq^m y_{m,a}-q_3^{m-2}z_b)(q_1\ssq^m y_{m,a}-q_3^{m-2}z_b)}
\\
&\times
\prod_{i=m+1}^{m+n-2}\prod_{1\le a\le L\atop L<b\le N}
d^{-1}
\frac{(q_1\ssq^{2m-i}y_{i,a}-q_2q_3^{2m-i-1}z_b)
(q_1\ssq^{2m-i}y_{i,a}-q_2^{-1}q_3^{2m-i-1}z_b)}
{(q_1\ssq^{2m-i}y_{i,a}-q_3q_3^{2m-i-1}z_b)
(q_1\ssq^{2m-i}y_{i,a}-q_3^{-1}q_3^{2m-i-1}z_b)}
\\
\times&
\prod_{1\le a\le L\atop L<b\le N}
\frac{(y_{m+n-1,a}-q_2q_3^{m-n}z_b)(y_{m+n-1,a}-q_2^{-1}q_3^{m-n}z_b)}
{(y_{m+n-1,a}-q_3q_3^{m-n}z_b)(y_{m+n-1,a}-q_1^{-1}q_3^{m-n}z_b)}
\frac{(y_{m+n-1,a}-q_2z_b)(y_{m+n-1,a}-z_b)}
{(y_{m+n-1,a}-q_3^{-1}z_b)(y_{m+n-1,a}-q_1^{-1}z_b)}\,.
\end{align*} 
For $n=1$, $B$ is the same and $A,C$ are modified as
\begin{align*}
A&= 
\prod_{i=0}^{m-1}\prod_{1\le a,b\le L}
\frac{y_{i,a}-y_{i+1,b}}{y_{i,a}-\ssq y_{i+1,b}}
\times
\prod_{1\le a< b\le L}
\frac{d^{-1}(y_{m,a}-y_{m,b})^2}
{(y_{m,a}-q_1y_{m,b})(y_{m,a}-q_1^{-1}y_{m,b})}\,,
\\
C&=\prod_{i=1}^{m-1}\prod_{1\le a\le L\atop L<b\le N}
d
\frac{(q_1\ssq^i y_{i,a}-q_2q_3^{i-1}z_b)(q_1\ssq^i y_{i,a}-q_2^{-1}q_3^{i-1}z_b)}
{(q_1\ssq^i y_{i,a}-q_3^{i}z_b)(q_1\ssq^i y_{i,a}-q_3^{i-2}z_b)}
\\
\times&
\prod_{1\le a\le L\atop L<b\le N}
\frac{(y_{m,a}-q_1^{-1}q_3^{m-2}z_b)(y_{m,a}-q_3^{m-1}z_b)}
{(y_{m-1,a}-q_3^{m-2}z_b)(y_{m,a}-q_1^{-1}q_3^{m-1}z_b)}
\frac{(y_{m,a}-q_2z_b)(y_{m,a}-z_b)}
{(y_{m,a}-q_3^{-1}z_b)(y_{m,a}-q_1^{-1}z_b)}\,.
\end{align*}

\medskip

Denoting $\bar{F}_{N,L}$ by $\pi_{N,L}(F)$,
we set  
$\pi_{m|n}^{m|n-1}=\sum_{L=0}^N\pi_{N,L}$.

\begin{prop}\label{prop:fusion-glmn}
The linear map  
\begin{align*}
&\pi_{m|n}^{m|n-1}:Sh^0_{m|n}(q_1,q_2,q_3)
\longrightarrow
Sh^0_{m|n-1}(\bar q_1,q_2,\bar q_3)
\otimes Sh_1(q_1^{-1}q_3^{-m+n-1},q_2^{-1},q_3^{m-n})\,,
\\
&\bar{q}_1=q_1 q_1^{-\frac{1}{m-n+1}}\,,
\quad \bar{q}_3=q_3 q_1^{\frac{1}{m-n+1}}\,,
\end{align*}
is a homomorphism of $\Z_{\ge0}$-graded algebras.
\end{prop}
\begin{proof}
 Let $F\in Sh_{m|n,M}^0$,  $G\in Sh_{m|n,N}^0$. 
First we examine the effect of the specialization 
\eqref{xtoy},\eqref{xtoz} 
in the shuffle product
\eqref{shfl}. 
Let $0\le L\le M+N$. 
For each partition 
$I_i\sqcup J_i=\{1,\ldots,M+N\}$ such that $|I_i|=M$, $|J_i|=N$, set
\begin{align*}
&I_i'=I_i\cap\{1,\ldots,L\}\,,\quad I_i''=I_i\cap\{L+1,\ldots,M+N\}\,,\\
&J_i'=J_i\cap\{1,\ldots,L\}\,,\quad J_i''=J_i\cap\{L+1,\ldots,M+N\}\,,
\end{align*}
so that 
\begin{align*}
&I'_i\sqcup J'_i=\{1,\ldots,L\}\,,\quad I''_i\sqcup J''_i=\{L+1,\ldots,M+N\}\,,
\\
&|I'_i|+|I''_i|=M\,,\quad |J'_i|+|J''_i|=N\,.
\end{align*}
In order that the factor 
\begin{align*}
\prod_{a\in I'_{m+n-1}\,, b\in J'_{m+n}}\omega_{m+n-1,m+n}(x_{m+n-1,a},q_1x_{m+n-1,b}) 
\end{align*}
be non-vanishing, we must have $I_{m+n-1}'\cap J'_{m+n}=\emptyset$, which implies
\begin{align*}
I'_{m+n-1}\subset I'_{m+n}\,. 
\end{align*}
Similarly, 
\begin{align*}
&\prod_{a\in I''_{i}\,, b\in J''_{i-1}}
\omega_{i,i-1} (q_3^{i-1}z_a,q_3^{i-2}z_b) \neq 0, \quad 1\le i\le m\,,
\\
&\prod_{a\in I''_{i}\,, b\in J''_{i-1}}
\omega_{i,i-1} (q_3^{2m-i-1}z_a,q_3^{2m-i}z_b) \neq 0, \quad m+1\le i\le m+n-1\,,
\end{align*}
implies 
\begin{align*}
I''_{m+n-1}\subset I''_{m+n-2}\subset\cdots\subset I''_1\subset I''_{m+n}\,. 
\end{align*}
Since $|I'_i|+|I''_i|=M$ for all $i$, we must have that
\begin{align*}
&I'_{m+n-1}=I'_{m+n}\,,\quad
I_1''=\cdots=I''_{m+n} \,.
\end{align*}
Writing $I''=I''_i$ and $J''=J''_i$, 
we see that $\rho_{M+N,L}\bigl(F*G\bigr)$ is reduced to the sum over partitions
\begin{align*}
&I_i'\sqcup J'_i =\{1,\ldots,L\}\,,\quad i=1,\ldots,m+n-1\,,\\
&I''\sqcup J''=\{L+1,\ldots,M+N\}\,,
\\
&|I_i'|=k\,,\quad |J_i'|=L-k\,,
\quad |I''|=M-k\,,\quad |J''|=N-L+k\,.
\end{align*}

It remains to show that, 
after cancelling out the factors $ABC$ in the Lemma from both sides, 
this sum reduces to the shuffle product 
$\pi_{M,k}(F)*\pi_{N,L-k}(G)$. 
This is a lengthy but straightforward computation.
\end{proof}
\bigskip

Iterating the map in Proposition \ref{prop:fusion-glmn}, we obtain 
\begin{cor}\label{cor:shmn}
There exists a homomorphism of $\Z_{\ge0}$-graded algebras 
\begin{align}
\pi_{m|n}: &Sh^0_{m|n}(q_1,q_2,q_3)
\longrightarrow Sh^0_{m}(q_1q_1^{-\frac{n}{m}},q_2,q_3q_1^{\frac{n}{m}})
\otimes\Bigl(
\bigotimes_{j=0}^{n-1}Sh_1(q_2^{j+1}q_3^{-m+n},q_2^{-1},q_2^{-j}q_3^{m-n})\Bigr)\,.\label{pimn}
\end{align}
\end{cor}
\bigskip

\subsection{Fusion homomorphism for $Sh_{m}$}\label{sec:fusion-m}

The fusion homomorphism can be defined also for $Sh_m$. 
Since the argument is quite similar, we state only the result.

Let $m\ge2$. 
As before, we fix $N,L$ with $0\le L\le N$. Consider the specialization
of
variables from $x_i=(x_{i,a})_{1\le a\le N}$, $1\le i\le m$, to
$y_i=(y_{i,a})_{1\le a\le L}$, $1\le i\le m-1$, and 
$z=(z_{a})_{L+1\le a\le N}$, given by
\begin{align}
&\text{$1\le a \le L$}:&&x_{i,a}=
q_1^{-\frac{m-i}{m-1}}y_{i,a}\quad  \text{for $1\le i\le m$},
\label{xtoy-2}
\\
&\text{$L+1\le b \le N$}:&&x_{i,b}=
q_3^{i-1}z_b\quad \text{for $1\le i\le m$}.
\label{xtoz-2}
\end{align}
Here we set $y_{m,a}=y_{1,a}$. 
Let $\tilde{\rho}_{N,L}(F)$ denote the substitution given by 
\eqref{xtoy-2} and \eqref{xtoz-2}. 
For $F\in Sh^0_{m,N}$ we define $\tilde{F}_{N,L}$ by 
\begin{align*}
\tilde{\rho}_{N,L}(F)&=\tilde{F}_{N,L}(y_1;\cdots;y_{m-1}|z) 
\\
&\times\prod_{1\le a<b\le L}
\Bigl(d
\frac{(y_{1,a}-q_2y_{1,b})(y_{1,a}-q_2^{-1}y_{1,b})}
{(y_{1,a}-q_1y_{1,b})(y_{1,a}-q_1^{-1}y_{1,b})}
\Bigr)
\times
\prod_{1\le a,b\le L}
\prod_{i=1}^{m-1}
\frac{y_{i,a}-y_{i+1,b}}{y_{i,a}-q_1^{1/(m-1)}y_{i+1,b}}
\\
&\times
\prod_{L+1\le a<b\le N}
\Bigl(d
\frac{(z_a-q_2z_b)(z_a-q_2^{-1}z_b)}{(z_a-q_3z_b)(z_a-q_3^{-1}z_b)}
\Bigr)^{m-1}
\prod_{L+1\le a,b\le N}
\frac{d(z_a-z_b)^2}
{(z_a-q_3^{m-1}z_b)(z_a-q_3^{-m+1}z_b)}
\\
&\times
\prod_{1\le a\le L\atop L+1\le b\le N}
\Bigl(\prod_{i=2}^{m-1}
d
\frac{(q_1^{-(m-i)/(m-1)}y_{i,a}-q_1q_3^i z_b)}
{(q_1^{-(m-i)/(m-1)}y_{i,a}-q_3^i z_b)}
\frac{(q_1^{-(m-i)/(m-1)}y_{i,a}-q_1^{-1}q_3^{i-2} z_b)}
{(q_1^{-(m-i)/(m-1)}y_{i,a}-q_3^{i-2} z_b)}
\Bigr)
\\
&\times
\prod_{1\le a\le L\atop L+1\le b\le N}
d^2
\frac{(y_{1,a}-q_1^2q_3z_b)(y_{1,a}-q_3^{-1}z_b)}
{(y_{1,a}-q_1q_3z_b)(y_{1,a}-z_b)}
\frac{(y_{1,a}-q_1^{-1}q_3^{m-2}z_b)(y_{1,a}-q_1q_3^{m}z_b)}
{(y_{1,a}-q_3^{m-2}z_b)(y_{1,a}-q_1q_3^{m-1}z_b)}\,.
\end{align*}

\begin{prop}\label{prop:fusion_glm}
The map sending $F\in Sh^0_{m,N}$ to $\sum_{L=0}^N\tilde{F}_{N,L}$
gives
a homomorphism of $\Z_{\ge0}$-graded algebras 
\begin{align*}
\pi_m^{m-1}: Sh^0_{m}(q_1,q_2,q_3)
\longrightarrow Sh^0_{m-1}(\tilde{q}_1,q_2,\tilde{q}_3)
\otimes Sh_1(q_1q_3^{-m+1},q_2,q_3^{m})\,,
\end{align*}
where
\begin{align}
&\tilde{q}_1=q_1q_1^{\frac{1}{m-1}} \,,
\quad \tilde{q}_3=q_3q_1^{-\frac{1}{m-1}}\,.
\label{tilde-q}
\end{align}
\end{prop}
\bigskip

\begin{cor}\label{cor:shm}
There exists a homomorphism of $\Z_{\ge0}$-graded algebras 
\begin{align}
\pi_m: Sh^0_{m}(q_1,q_2,q_3)
\longrightarrow 
\bigotimes_{j=0}^{m-1}  Sh_{1}(q_2^jq_1^m,q_2,q_2^{-j-1}q_1^{-m})\,.
\label{pim}
\end{align}
\end{cor}

\medskip

\section{Subalgebra $\cB_{m|n}(s)$}\label{Bethe sec}

\subsection{Subalgebra $\cB_{m|n}(s)$}\label{subsec:cB}

We introduce a family of subspaces $\cB_{m|n}(s)$ of $Sh_{m|n}$
defined by certain regularity conditions at $0$ and $\infty$ 
depending on parameters $s=(s_1,\dots,s_{m+n})$. 
We will show in Section \ref{G functions sec}
that for each $s$,  $\cB_{m|n}(s)$ is a free commutative subalgebra of $Sh_{m|n}$. We  call $\cB_{m|n}(s)$ the Bethe algebra. 
We expect that in relation to the quantum toroidal algebras, 
$\cB_{m|n}(s)$ is a deformation of the Cartan half currents of the affine $U_q\hat\gl_{m|n}\in\mc E_{m|n}$ which corresponds to  $s_1,\ldots,s_{m+n-1}\to 0$,
cf. Remark 3.11 [FT].

To this end, we prepare some notation.

Quite generally, for an element $F\in Sh_{m|n,\bN}$ and $\bs k=(k_1,\ldots,k_{m+n})\in\Z_{\geq 0}^{m+n}$
with $0\le k_i\le N_i$,
we set 
\begin{align*}
&F_\xi^{\bk} =F(\xi x_1',x_1'';\cdots;\xi x_{m+n}',x_{m+n}'')\,,
\end{align*}
where
\begin{align*}
&x_i'=(x_{i,1},\ldots,x_{i,k_i})\,,\quad 
x_i''=(x_{i,k_i+1},\ldots,x_{i,N_i})\,.
\end{align*}
We shall write
\begin{align}
&F_\infty^{\bk} =\lim_{\xi\to\infty}F_\xi^{\bk}\,,
\quad
F_0^{\bk} =\lim_{\xi\to0}F_\xi^{\bk}\,,
\label{limIZ}
\end{align}
provided these limits exist.

Further we fix an $(m+n)$-tuple of parameters
\begin{align*}
s=(s_1,\ldots,s_{m+n})\in(\C^{\times})^{m+n}
\quad \text{satisfying $\prod_{i=1}^{m+n}s_i=1$}.   
\end{align*}
We extend the suffix to all $i\in\Z$ by $s_{i+m+n}=s_i$. 
We shall say that $s$ is {\it generic} if 
\begin{align*}
q_1^{a_1}q_2^{a_2}\prod_{i=1}^{m+n-1}s_i^{b_i}=1\,,  
a_i,b_j\in\Z\,, 
\quad 
\text{holds only if $a_i=b_j=0$ for all $i,j$}.
\end{align*}

\begin{dfn}
We define the subspace  $\cB_{m|n,N}(s)$ to be the set of all elements
$F\in Sh^0_{m|n,N}$ such that both limits \eqref{limIZ} exist  
for all $\bk=(k_1,\ldots,k_{m+n})$ satisfying $0\le k_i\le N$,
and are related by
\begin{align}\label{ratio of limits}
F^{\bk}_\infty=F^{\bk}_0 \times 
\prod_{i=1}^{m+n} s_i^{k_i} \cdot d^{2N(k_m-k_{m+n})}\,.
\end{align}
\end{dfn}
Note that the power of $d$ is absent if $n=0$. 
\medskip

We set
\begin{align*}
\cB_{m|n}(s)
&=\bigoplus_{N=0}^\infty \cB_{m|n,N}(s)\,.
\end{align*}
Examples of elements in $\cB_{m|n}(s)$ will be given in 
Section \ref{sec:GrN} below.

\begin{lem}
The subspace $\cB_{m|n}(s)$ is a subalgebra of $Sh^0_{m|n}$.
\end{lem}
\begin{proof}
Consider the shuffle product \eqref{shfl} of
$F\in \cB_{m|n,M}(s)$ and $G\in \cB_{m|n,N}(s)$. In each summand
denote
\begin{align*}
&F_I=F(x_{1,I_1};\cdots;x_{m+n,I_{m+n}})\,,
\quad 
G_J=G(x_{1,J_1};\cdots;x_{m+n,J_{m+n}})\,,
\\
&\omega_{I,J}=\prod_{i,j=1}^{m+n}\prod_{a\in I_i\atop b\in J_j}
\omega_{i,j}(x_{i,a},x_{j,b})\,.
\end{align*}
Let $h_i=|I_i\cap \{1,\ldots,k_i\}|$, $l_i=|J_i\cap \{1,\ldots,k_i\}|$,
so that $h_i+l_i=k_i$.
Then 
\begin{align}
&(F_I)^{\bk}_\infty= (F_I)^{\bk}_0\cdot \prod_{i=1}^{m+n}s_i^{h_i}\cdot 
d^{2M(h_m-h_{m+n})}\,,
\label{BmnF}
\\
&(G_J)^{\bk}_\infty= (G_J)^{\bk}_0
\cdot \prod_{i=1}^{m+n}s_i^{l_i}\cdot d^{2N(l_m-l_{m+n})}\,.
\label{BmnG}
\end{align}
On the other hand, it is easy to see that
\begin{align*}
&\lim_{\xi\to\infty}\omega_{i,i}(\xi x,y)
=\lim_{\xi\to0}\omega_{i,i}(\xi x,y)\cdot  
q_2^{-\theta(1\le i\le m-1)+\theta(m+1\le i\le m+n-1)}\,,
\\
&\lim_{\xi\to\infty}\omega_{i,i+1}(\xi x,y)
=\lim_{\xi\to0}\omega_{i,i+1}(\xi x,y) \cdot
q_1^{-\theta(0\le i\le m-1)+\theta(m\le i\le m+n-1)}\,,
\\
&\lim_{\xi\to\infty}\omega_{i,i-1}(\xi x,y)
=\lim_{\xi\to0}\omega_{i,i-1}(\xi x,y) \cdot
q_3^{-\theta(1\le i\le m)+\theta(m+1\le i\le m+n)}\,,
\end{align*}
and that
\begin{align*}
\frac{\lim_{\xi\to\infty}\omega_{i,j}(x,\xi y)}
{\lim_{\xi\to0}\omega_{i,j}(x,\xi y) }
=
\Bigl(\frac{\lim_{\xi\to\infty}\omega_{i,j}(\xi x,y)}
{\lim_{\xi\to0}\omega_{i,j}(\xi x,y) }\Bigr)^{-1}
\,.
\end{align*}
This yields
\begin{align}
(\omega_{I,J})^{\bk}_\infty =(\omega_{I,J})^{\bk}_0 \cdot
d^{2N(h_m-h_{m+n})+2M(l_m-l_{m+n})}\,.
\label{BmnO}
\end{align} 
It follows from \eqref{BmnF}, \eqref{BmnG} and \eqref{BmnO} that
\begin{align*}
(F*G)^{\bk}_\infty =(F*G)^{\bk}_0\cdot 
\prod_{i=1}^{m+n}s_i^{k_i}\cdot d^{2(M+N)(k_m-k_{m+n})}\,.
\end{align*}
This completes the proof.  
\end{proof}
\medskip

We call the subalgebra $\cB_{m|n}(s)$ the Bethe 
algebra.

\begin{rem}
For $n=0$, the subspace $\cB_{m|0}(s)$ is a slight modification of 
$\cA(\bar{s})$ studied in \cite{FT}.
Though they turn out to coincide for generic $s$, 
it is a priori not obvious that the latter is a subalgebra. \qed
\end{rem}

\subsection{Gordon filtration}\label{sec:Gordon}

Gordon filtration is a standard method used to study spaces of polynomials with vanishing conditions on the planes of codimension higher than one. It is designed in such a way that associated graded spaces become spaces of polynomials without conditions. 

In this section we apply the Gordon filtration technique to 
obtain an estimate from above to the dimensions of the graded components $\mc B_{m|n,N}(s)$  of 
the Bethe algebra. 
Our proof is similar to the proof of Lemma 2.4 in \cite{FT}. 

Let $\mc R_{m|n}$ denote the graded polynomial algebra in variables $\{T_{i,\nu}\ |\ 1\leq i\leq m+n,\nu\in\Z_{>0}\}$ with $\deg T_{i,\nu}=\nu$. Let $\mc R_{m|n,N}$ denote the degree $N$ component of $\mc R_{m|n}$.

\begin{prop}\label{prop:Gordon}
Assume $s$ is generic and $m\neq n$. Then $\dim \mc B_{m|n,N}(s)\leq \dim \mc R_{m|n,N}$.
\end{prop}
\begin{proof}
Fix $N\in\Z_{\geq 0}$ and denote  $B=\mc B_{m|n,N}$ and $Sh=Sh^0_{m|n,N}$.

We will define a finite partially ordered set $P$ and for each $p\in P$ we will give a subspace $B_p\subset B$ and show that $\cap_p B_p=0$. That will produce a bound:
$$
\dim B\leq \sum_{p\in P} \dim \on{gr} B_p, \qquad \on{gr} B_p= \Big(\mathop{\cap}_{p'>p} B_{p'}\Big)/\Big( \mathop{\cap}_{p'\geq p} B_{p'}\Big).
$$
Here, if $p\in P$ is a maximal element, then we set $\mathop{\cap}_{p'>p} B_{p'}=B$.

Indeed, refine the partial order $>$ on $P$ to a total order $\succ$ and set 
$\tilde B_p=\mathop{\cap}_{p'\succ p}  B_{p'}$.
Then $p\succ p'$ implies $\tilde B_p\supset \tilde B_{p'}$. And, therefore, we have a linear  filtration $\tilde B_p$ on $B$ and therefore as a vector space 
$$
B\simeq\oplus \on{gr} \tilde B_p,\qquad \on{gr} \tilde B_p=\Big(\mathop{\cap}_{p'\succ p} B_{p'}\Big)/\Big( \mathop{\cap}_{p'\succeq p} B_{p'}\Big).
$$

For any subspaces $A_1,A_2,C$, with $A_2\subset A_1$, we have an injection $(A_1\cap C)/(A_2\cap C)\to A_1/A_2$. 
Therefore for every $p$ we have an injective map 
$\on{gr} \tilde B_p\to  \on{gr} B_p$ which corresponds to 
$A_1=\mathop{\cap}_{p'> p} B_{p'}$, 
$A_2=\mathop{\cap}_{p'\geq p} B_{p'}$,
$C=\mathop{\cap}_{p'\succ p} B_{p'}.$ Thus we have the upper bound on the dimension of $B$.

After that we will argue that $\on{gr} B_p$ are at most one dimensional, and in many cases are trivial which will result in the proof of the proposition.

\medskip

We start by defining the set $P$. 

For $a,b\in\Z$, denote the interval $[a,b]=\{a,a+1,\dots,b\}\subset\Z$. The interval $[a,b]$ is non-empty if and only if $a\leq b$. For $i=1,\dots,m+n$ denote $[a,b]_i$ the number of elements of $[a,b]$ congruent to $i$ modulo $m+n$. Thus modulo $m+n$ we have 
$$
[a,b]\equiv \{\underbrace{1,\dots,1}_{[a,b]_1},\underbrace{2,\dots,2}_{[a,b]_2},\dots,\underbrace{m+n,\dots,m+n}_{[a,b]_{m+n}} \}.
$$
Intervals $[a,b]$ and $[c,d]$ are called equivalent if they are shifts by a multiple of $m+n$:  $a-c=b-d=(m+n)k$, for some $k\in\Z$. If $[a,b]$ and $[c,d]$ are equivalent then $[a,b]_i=[c,d]_i$ for all $i$. The converse is not true, for example $[1,m+n]$ and $[2,m+n+1]$ are not equivalent but $[1,m+n]_i=[2,m+n+1]_i=1$ for all $i$.

An interval $[a,b]$ is called cyclic if its number of elements is a multiple of $m+n$: $b-a+1=(m+n)k$, $k\in\Z_{>0}$. In other words, an interval $[a,b]$ is cyclic if and only if $[a,b]_i$ does not depend on $i$.

For $\bs a=(a_1,\dots,a_s)$, $\bs b=(b_1,\dots,b_s)$, $a_j\leq b_j$, denote $[\bs a,\bs b]$ the unordered set of intervals $[a_j,b_j]$.  

The set of intervals $[\bs a,\bs b]$ is called $N$-admissible if 
intervals $[a_j,b_j]$ are non-empty for all $j$ and $\sum_{j=1}^s[a_j,b_j]_i=N$ for all $i$.

Two sets of intervals $[\bs a,\bs b]$ and $[\bs c,\bs d]$ are called equivalent if there exists an ordering of intervals in $[\bs a,\bs b]$ such that $[a_j,b_j]$ is equivalent to $[c_j,d_j]$ for all $j$. 

The set of intervals $[\bs a,\bs b]$ is called cyclic if intervals $[a_j,b_j]$ are cyclic for all $j$. Note that $\dim \mc R_{m|n,N}$ equals the number of equivalence classes of $N$-admissible cyclic sets of intervals.

The set $P$ is the set of all equivalence classes of $N$-admissible intervals $[\bs a,\bs b]$.

\medskip

Next we define an order on $P$. 

A partition of $M$ is an ordered set of non-negative integers called parts which sum up to $M$: $\la=(\la_1,\dots,\la_s)$, $\la_1\geq \dots\geq \la_s\geq 0$,  $\sum_{j=1}^s \la_j=M$.  Define the lexicographical partial ordering on set of partitions:
$\la>\mu$ if for some $j$, $\la_j>\mu_j$ and $\la_1=\mu_1$, $\la_2=\mu_2$, $\dots$, $\la_{j-1}=\mu_{j-1}$.

Given a set of intervals $[\bs a,\bs b]$, let $\la_{[\bs a,\bs b]}$ be the partition of $(m+n)N$ with parts $b_j-a_j+1$ written in the non-decreasing order. We give a partial order on $P$ according to the order on the corresponding partitions: we set $[\bs a,\bs b]>[\bs c, \bs d]$ if   $\la_{[\bs a,\bs b]}>\la_{[\bs c, \bs d]}$. If  $\la_{[\bs a,\bs b]}=\la_{[\bs c, \bs d]}$ and $[\bs a,\bs b]\neq [\bs c, \bs d]$, then $[\bs a,\bs b]$ and $[\bs c, \bs d]$ are not comparable.

\medskip

Now we define the subspaces $B_{[\bs a,\bs b]}$, $[\bs a,\bs b]\in P$. 
Given $[\bs a,\bs b]\in P$, we define the substitution of variables $(x_{1,1},\dots,x_{1,N},\dots,x_{m+n,1},\dots, x_{m+n,N})$ appearing in $Sh$ to the variables $y_1,\dots, y_s$ multiplied by some constants as follows.

For $a\in \Z$ define recursively constants $\al_a$ by the rule $\al_0=1$, $\al_{a+1}=q_3\al_a$ if $a\equiv i$ (mod $m+n$) with $0\leq i\leq m-1$ and  $\al_{a+1}=q_3^{-1}\al_a$ if $a\equiv i$ (mod $m+n$) with $m\leq i\leq m+n-1$. Thus $\al_a$ is a power of $q_3$, and $\al_{k(m+n)}=q_3^{(m-n)k}$. In particular, since $m\neq n$,  $a\equiv b$ together with $a\neq b$ implies $\al_a\neq \al_b$.

Since $[\bs a,\bs b]\in P$ is $N$ admissible, there exists a bijection $\tau$ between the set $[1,m+n]\times[1,N]$ 
with the
union of intervals $[a_j,b_j]$:
$$
\tau:\ \{1,\dots,m+n\} \times \{1,\dots, N\} \to \sqcup_{j=1}^s \{a_j,a_j+1,\dots, b_j\}.
$$
Let $\tau_1(i, j) \in\{1,\dots, s\}$ be the label of the interval of $\tau(i\times j)$, and let  $\tau_2(i, j)\in [a_{\tau_1(i, j)},b_{\tau_1(i, j)}]$ be the number in that interval
given by  $\tau(i\times j)$. Thus $\tau(i\times j)=\tau_2(i,j)\in [a_{\tau_1(i,j)},b_{\tau_1(i,j)}]$.

We choose $\tau$ in such a way that the color is preserved: $\tau_2(i,j)\equiv i\ (\on{mod} m+n)$. Clearly such a bijection exists, and it is not unique in general.

We view $\tau$ as a substitution of variables $x_{i,j}$ appearing in $Sh$ substituting $x_{i,j}$ to $\al_{\tau_2(i, j)}\,y_{\tau_1(i, j)}$.

Recall that functions $F\in Sh$ have the form
$$
F=\frac{f(x_1;\dots;x_{m+n})}{\Pi}\qquad \Pi={\prod_{i=1}^{m+n}\prod_{r,l=1}^N(x_{i,r}-x_{i+1,l})},
$$
where the numerator $f$ is a Laurent polynomial of variables $x_{i,j}$ symmetric in bosonic variables with the same first index, skew-symmetric in the fermionic variables with the same first index and satisfying the wheel conditions.

If $F\in\B$, then the Laurent polynomial $f$ is actually a polynomial since the limit as $x_{i,j}\to 0$ exists.
The total degree of polynomial $f$ is the same as the total degree of $\Pi$, that is $(m+n)N^2$ since the limits when all variables go to infinity, $F_\infty^{(N,\dots,N)}$, or to zero 
$F_0^{(N,\dots,N)}$, are well-defined. In each variable the degree is at most $2N$ since the limit when this variable goes to infinity is well-defined.

Define an evaluation map  
$$
\phi_{[\bs a,\bs b]}: \  Sh \to \C[y_1,\dots,y_s] , \qquad 
 F(x_{i,j})\mapsto f(\tau(x_{i,j})),
$$
substituting all variables $x_{i,k}$ in the numerator of $F\in Sh$ to $y_j$ according to the substitution $\tau$.

Note that because of the symmetry of polynomial $f\in Sh$ the result of substitution up to a sign is independent of the choice of the bijection $\tau$. In particular the kernel of the evaluation map $\Ker \phi_{[\bs a,\bs b]}\subset Sh$ does not depend on the choice of $\tau$.

Note that if $[\bs a,\bs b]$ and $[\bs c,\bs d]$ are equivalent:
$a_j-c_j=b_j-d_j=Nk_j$ then $\phi_{[\bs a,\bs b]}$ coincides with $\phi_{[\bs c,\bs d]}$ followed by the shift  $y_j\to q_3^{k_j(m-n)}y_j$ for all $j$. In particular, $\Ker \phi_{[\bs a,\bs b]}= \Ker \phi_{[\bs c,\bs d]}$.

Note that $m\neq n$ is important here since $F\in Sh$ vanishes if two fermionic variables are set equal and therefore the evaluation map would be zero in the case when one of the intervals $[a_j,b_j]$ contains two numbers congruent to $m$ or  to $m+n$. 

Note also that (unlike the purely bosonic case), in some cases, the evaluation applied to the denominator $\Pi$ can give a zero.

Define 
$$
B_{[\bs a ,\bs b]}=\on{Ker} \phi_{[\bs a, \bs b]}, \qquad 
\bar B_{[\bs a ,\bs b]}=\mathop{\cap}_{[\bs c ,\bs d]>[\bs a ,\bs b]}B_{[\bs c ,\bs d]}.
$$

\medskip
We claim that $\cap_{[\bs a,\bs b]}B_{[\bs a ,\bs b]}=0$. Indeed, let $F=f/\Pi\in Sh$ vanish under all evaluation maps. In particular, it has to vanish at $x_{i+1,r}=q_3x_{i,l}$ if $i$ is before the equator and at $x_{i+1,r}=q_3^{-1}x_{i,l}$ if $i$ is after the equator. Thus the polynomial $f$ of  degree $(m+n)N^2$ has $(m+n)N^2$ linear factors and is determined up to a constant. Thus
$$
F=\frac{c}{\Pi} \ \prod_{i=0}^{m-1}\prod_{r,l=1}^N(x_{i+1,r}-q_3x_{i,l})\prod_{i=m}^{m+n-1}\prod_{r,l=1}^N(x_{i+1,r}-q_3^{-1}x_{i,l}),
$$
where $c$ is a constant. But such $F$ is clearly not in $B$.
\medskip

Next we study $\on{gr} B_{[\bs a ,\bs b]}=\phi_{[\bs a ,\bs b]}(\bar B_{[\bs a ,\bs b]})$. We now show that this space is at most one dimensional.

Denote $g=\phi_{[\bs a ,\bs b]}(F)$. Recall that $\deg g=\deg f=(m+n)N^2$.
We claim that $g$ has $(m+n)N^2$ linear  
factors, one for each factor of $\Pi$, and therefore it is unique up to a constant which implies $\on{gr}B_p$ is one-dimensional.

Let us say a variable $x_{i,r}$ has type $j$ if it specializes to a multiple of $y_j$: $\tau_1(i,r)=j$.

First, $g$ is divisible by $y_j^{r_j}$ where $r_j$ is the number of factors $x_{i,r}-x_{i+1,l}$ in $\Pi$ such that both 
$x_{i,r}$ and $x_{i+1,l}$ have type $j$.
This is because the limit of $F$ as all variables of type $j$ go to zero should be well defined. 

Next we consider linear factors of $g$ of the form $y_j-\al y_{j'}$, $j\neq j'$ and $\al\in q_1^\Z q_3^\Z$.  For that we fix an order $y_1,\dots,y_s$ in such a way that the corresponding intervals are ordered: $b_i-a_i\geq b_{i+1}-a_{i+1}$. 

Consider a factor  $x_{i,r}-x_{i+1,l}$ in $\Pi$ where the variables $x_{i,r}$, $x_{i+1,l}$ are specialized to two different variables.  
Suppose $\tau(i\times r)$ is $c\in[a_j,b_j]$ 
and $\tau((i+1)\times l)$ is $c'\in [a_{j'},b_{j'}]$. We present a corresponding zero as follows.

If $j>j'$  and $c'$ is bosonic, $c'\not\equiv m,m+n$ (mod $m+n$), and $c<b_j$, then we obtain a zero from the cubic wheel condition corresponding to variables $x_{i,r},x_{i+1,l}$  and the variable which corresponds to $(c+1)\in [a_j,b_j]$. It guarantees that $g$ is divisible by $y_j-\al y_{j'}$ with $\al= q_1^{\pm 1}\al_{c'}/\al_c$. Here we have $q_1$ if $c'$ is before the equator, 
that is $c'\equiv 1,\dots,m-1$ modulo $m+n$, and $q_1^{-1}$ if $c'$ is  after equator, that is 
$c'=\equiv m+1,\dots,m+n-1$ (mod $m+n$). 

In the case $c=b_j$, we obtain a factor $y_j-\al y_{j'}$ with 
$\al=q_3^{\pm1}\al_{c'}/\al_c $ due to the fact that $f$ should vanish under substitutions corresponding to larger sets of intervals in the sense of the order in $P$.

If $j'>j$  and $c$ is bosonic, $c\not\equiv m,m+n$ and $c'>a_{j'}$, then we similarly obtain a zero from the cubic wheel condition corresponding to variables $x_{i,r},x_{i+1,l}$  and the variable which corresponds to $(c'-1)\in [a_{j'},b_{j'}]$. In the case $c'=a_{j'}$, we again obtain a factor $y_j-\al y_{j'}$ with 
$\al=q_3^{\pm1}\al_{c'}/\al_c $ due to the fact that $f$ should vanish under larger substitutions.

Let $j>j'$  and let $c'$ be fermionic, $c'\equiv m$ or $c'\equiv m+n$.  We have three cases.

If $c+1<b_j$ then we use the quartic wheel condition involving  $x_{i,r},x_{i+1,l}$ and the variables which correspond to $c+1, c+2\in [a_j,b_j]$. Moreover, in this case we also have a factor coming from the skew-symmetry with respect to  $x_{i+1,l}$ and the variable corresponding to $c+1\in [a_j,b_j]$. This zero corresponds to the factor of the denominator $\Pi$ which is a difference of $x_{i+1,l}$ and the variable which corresponds to $c+2\in [a_j,b_j]$. 

If $c+1=b_j$, we get the factor from the skew-symmetry with respect to  $x_{i+1,l}$  and the variable which corresponds to $c+1\in [a_j,b_j]$.

If $c=b_j$, we get a factor due to the fact that $f$ should vanish under larger substitutions.

This covers all factors in the denominator. It remains to point out 
that some of the linear factors we discussed coincide. However, all coinciding factors of the form $y_j-\al y_{j'}$, $j'<j$, come from using conditions each of which uses a single variable of type $j'$. Moreover, all these variables of type $j'$ are distinct. Therefore, using the derivatives with respect to these variables, we see that 
we get $(y_j-\al y_{j'})^t$ where $t$ is the number of times this factor appeared in our construction.

\medskip

Finally we claim the  $\on{gr} B_{[\bs a ,\bs b]}$ is trivial unless $[\bs a ,\bs b]$ is cyclic.
Indeed, fix $j$ and consider the limits of $F\in B\subset Sh$ which sends variables of type $j$ to zero or to infinity. Substitute $x_{i,r}\to \xi x_{i,r}$ in $F$  for variables $x_{i,r}$ of type $j$.
Then we have factors in the denominator $\Pi$ which do not depend on $\xi$. These are binomials, where variables either both have type $j$ or both do not have type $j$. These factors cancel when we take the ratio of the limits $\xi\to 0$ and $\xi\to \infty$. Therefore we can remove these factors. After that we proceed by applying the evaluation as in $\phi_{[\bs a,\bs b]}$. Take now the ratio of the limits $\xi\to 0$ and $\xi\to \infty$.
Then the result is a possible monomial in 
$y_j^{\pm1}$ 
with a possible coefficient $q_1^a q_3^b$, $a,b\in\Z$. By definition, since $F\in B$, this ratio equals to a product of $s_i$ multiplied by a power of $d$ as in \eqref{ratio of limits}.
 Since $s$ is generic, it follows that the length of $j$-th interval has to be a multiple of $m+n$.

\medskip

Combining, we see that the dimension of $B$ is at most the number of non-empty $N$-admissible cyclic sets of intervals which is exactly $\dim \mc R_{m|n,N}$.
\end{proof}

\medskip

\begin{rem}
The same  proof shows that if there exists a non-zero $F\in Sh_{m|n,\bs N}$
such that both limits \eqref{limIZ} exist for all $\bs k=(k_1,\dots,k_{m+n})$ satisfying $0\leq k_i\leq N_i$ and such that \eqref{ratio of limits} holds, then $\bs N=(N,\dots, N)$ for some $N\in\Z_{\geq 0}$. 

In other words, the Bethe algebra indeed belongs to $Sh^0_{m|n}$.\qed
\end{rem}

\section{Generators and commutativity of $\cB_{m|n}(s)$} \label{G functions sec}

\subsection{Functions $I^c_{M,N}$}\label{sec:IcN}

Before proceeding, we prepare and study
some auxiliary rational functions which will be used
in the next subsection to construct a family of
elements of the subalgebra $\cB_{m|n}(s)$. 

For variables $y=(y_1,\ldots,y_M)$ and $z=(z_1,\ldots,z_N)$,  
we shall use the notation
\begin{align*}
&\Delta(y)=\prod_{1\le a<b\le M}(y_a-y_b)\,,
\quad \bar{y}=\prod_{a=1}^M y_a\,,
\quad \Pi(y,z)=\prod_{a=1}^M\prod_{b=1}^N(y_a-z_b)\,,
\end{align*}
and similarly for $\Delta(z),\bar{z}$. We shall use also
$[n]=(q^n-q^{-n})/(q-q^{-1})$.  

For $c\in\Z$, define rational functions $I^c_{M,N}(y,z)$ by 
\begin{align}
I^c_{M,N}(y,z)&= 
\frac{1}{\Delta(y)}
\prod_{a=1}^M\Bigl(\frac{q^cT_{q,y_a}-q^{-c}T^{-1}_{q,y_a}}{q-q^{-1}}\Bigr)
\Bigl(\frac{\Delta(y)}{\Pi(y,z)}
\Bigr)\,,
\label{Ic}
\end{align}
where 
$(T_{q,y_a}f)(y_1,\ldots,y_M)=f(y_1,\ldots,q y_a,\ldots,y_M).$ 

The functions $I^c_{M,N}(y,z)$ are clearly symmetric in $y_1,\dots,y_M$ and in $z_1,\dots,z_N$.

\begin{lem}\label{lem:Ic-prop}
Functions $I^c_{M,N}(y,z)$ enjoy the following properties: 
\begin{align}
&I^c_{M,N}(y,z)=
(-1)^{MN+N}\prod_{i=0}^{M-N-1}[c+i]\cdot I^{1-c}_{N,M}(z,y)
\quad \text{if $M\ge N$}\,,
\label{pr1}
\\
&I^c_{M,N}(y,z)=0\quad\text{if $M>N$ and 
$N-M+1\le c\le 0$},
\label{pr2}
\\
&\bar{y} \cdot I^{1}_{N,N}(y,z)=\bar{z}\cdot I^{0}_{N,N}(y,z)\,.
\label{pr3}
\end{align}
\end{lem}
\begin{proof}
Let $s_\lambda(y)$ denote the Schur polynomial
associated with a partition $\lambda$:
\begin{align*}
\Delta(y)s_\lambda(y) 
=\det\Bigl(y_a^{\lambda_b+M-b}\Bigr)_{1\le a,b\le M}\,.
\end{align*}
Here $\lambda=(\lambda_1,\ldots,\lambda_\ell,0,0,\ldots)$, 
$\lambda_1\ge\cdots\ge\lambda_\ell>0$, $\ell=\ell(\lambda)$.

Using the Cauchy identity we can write
\begin{align*}
\frac{\Delta(y)}{\Pi(y,z)}
&=(-1)^{MN}\bar{z}^{-M}
\frac{\Delta(y)}{\prod_{1\le a\le M\atop 1\le b\le N}(1-y_az_b^{-1})} 
\\
&=(-1)^{MN}\bar{z}^{-M}\sum_{\lambda}\Delta(y)s_\lambda(y)
s_\lambda(z^{-1})\,,
\end{align*}
where 
$z^{-1}=(z_1^{-1},\ldots,z_N^{-1})$,
and 
the sum is over all partitions 
$\lambda$ with length $\ell(\lambda)\le\min(M,N)$.
The determinant expression of $\Delta(y)s_\lambda(y)$ readily yields
\begin{align*}
\prod_{a=1}^M\Bigl(\frac{q^cT_{q,y_a}-q^{-c}T^{-1}_{q,y_a}}{q-q^{-1}}
\Bigr)
\Bigl(\Delta(y)s_\lambda(y) \Bigr)
=\prod_{a=1}^M[\lambda_a+M-a+c]\cdot \Delta(y)s_\lambda(y)\,,
\end{align*}
from which we obtain 
\begin{align}
I^c_{M,N}(y,z)
&= (-1)^{MN}\bar{z}^{-M}\sum_{\lambda}\prod_{a=1}^M
[\lambda_a+M-a+c]\,s_\lambda(y)s_{\lambda}(z^{-1})\,.
\label{Schur1}
\end{align}
By a similar calculation we have also
\begin{align}
I^c_{N,M}(z,y)
&= 
(-1)^N\bar{z}^{-M}\sum_{\lambda}\prod_{a=1}^N
[\lambda_a+M-a+1-c]\,s_\lambda(z^{-1})s_{\lambda}(y)\,.
\label{Schur2}
\end{align}
Comparing \eqref{Schur1} and \eqref{Schur2} we obtain \eqref{pr1}.

Property \eqref{pr2} is an immediate consequence of \eqref{pr1}. 

To see \eqref{pr3}, choose $M=N$ and $c=0$ in \eqref{Schur1}. 
Then only partitions with $\lambda_N>0$ contribute to the sum.
Rewriting $\lambda_a$ as $\lambda_a+1$ for $1\le a \le N$ and using
$s_{\lambda+(1,\ldots,1)}(y)=\bar{y}s_\lambda(y)$, we obtain 
\eqref{pr3}.
\end{proof}

\begin{lem}\label{lem:asy-Ic}
For $0\le k\le M$ and $0\le l\le N$, consider the substitution
\begin{align*}
&(y,z)=(\xi y',y'',\xi z',z'')\,,
\end{align*}
where
\begin{align*}
&y'=(y_1,\ldots,y_k)\,,\quad y''=(y_{k+1},\ldots,y_M)\,,\\
&z'=(z_1,\ldots,z_l)\,,\quad z''=(z_{l+1},\ldots,z_N)\,.
\end{align*}
Then the following asymptotics hold:
\begin{align*}
&I^c_{M,N}(\xi y',y'',\xi z',z'')
\\
&=\frac{(-1)^{(M-k)l}\xi^{-kN-(M-k)l}}{(\bar{y'})^{N-l}(\bar{z'})^{M-k}}
\Bigl(I^{c+M-k-N+l}_{k,l}(y',z')I^c_{M-k,N-l}(y'',z'')+O(\xi^{-1})
\Bigr)
\quad \text{as $\xi\to\infty$}\,,
\\ 
&
=\frac{(-1)^{k(N-l)}\xi^{-kl}}{(\bar{y''})^{l}(\bar{z''})^{k}}
\Bigl(I^{c}_{k,l}(y',z')I^{c+k-l}_{M-k,N-l}(y'',z'')+O(\xi)
\Bigr)
\quad
\text{as $\xi\to 0$}\,.
\end{align*}
\end{lem}
\begin{proof}
This follows from a direct calculation, noting
\begin{align*}
&\Delta(\xi y',y'')
=\xi^{k(k-1)/2} \Delta(y')\Delta(y'')
\times
\begin{cases}
\xi^{k(M-k)}\Bigl(\bigl(\bar{y}'\bigr)^{M-k}+O(\xi^{-1}) \Bigr)
& \text{as $\xi\to\infty$},
\\
(-1)^{k(M-k)}\bigl(\bar{y}''\bigr)^k+O(\xi)
& \text{as $\xi\to 0$},
\end{cases}
\\
&\Pi(\xi y',y'', \xi z',z'')
=\xi^{kl}\Pi(y',z')\Pi(y'',z'')
\\
&\times
\begin{cases}
\xi^{k(N-l)+(M-k)l}
\Bigl((-1)^{(M-k)l}
\bigl(\bar{y}'\bigr)^{N-l}\bigl(\bar{z}'\bigr)^{M-k}
+O(\xi^{-1})
\Bigr)
 & \text{as $\xi\to\infty$},
\\
(-1)^{k(N-l)}
\bigl(\bar{y}''\bigr)^l\bigl(\bar{z}''\bigr)^k
+O(\xi)
& \text{as $\xi\to 0$}.
\end{cases}
\end{align*} 
\end{proof}

The following is a corollary of Lemma \ref{lem:asy-Ic}.
\begin{lem}\label{lem:Ic-nonzero}
Let $M=N$ and $c>0$. 
Then the limit
\begin{align*}
\lim_{\xi\to\infty} \xi^{N(k+l)-kl}I^c_{N,N}(\xi y',y'',\xi z',z'')
\end{align*} 
exists. This limit is non-zero only if $c> k-l$.
Similarly the limit
\begin{align*}
\lim_{\xi\to 0} \xi^{kl}I^c_{N,N}(\xi y',y'',\xi z',z'')
\end{align*} 
exists. This limit is non-zero only if $c> -k+l$.
\end{lem}
\begin{proof}
By \eqref{pr2}, $I^{c-k+l}_{k,l}(y',z')=0$ if $k>l$ and 
$1\le c\le k-l$. 
Similarly  $I^{c+k-l}_{N-k,N-l}(y'',z'')=0$ if $k<l$ and 
$1\le c\le l-k$.
\end{proof}
\medskip

Clearing the denominator we define 
\begin{align}
&\tilde{I}^c_{M,N} (y,z)=
(q-q^{-1})^M
\prod_{1\le a\le M \atop 1\le b\le N}(q y_a-z_b)(q^{-1} y_a-z_b)\cdot 
I^c_{M,N}(y,z)\,.
\label{Itilde}
\end{align}
Then $\tilde{I}^c_{M,N} (y,z)$ 
is a polynomial in $y,z$, symmetric in $y$, symmetric in $z$,
and of total degree $MN$. 
It is also a Laurent polynomial in $q^c$:
\begin{align}
&\tilde{I}^c_{M,N}(y,z)=
\sum_{l=0}^M(-1)^l q^{(M-2l)c}\tilde{I}_{M,N,l}(y,z)\,,
\label{Itilde2}\\
&\tilde{I}_{M,N,l}(y,z)=
q^{(M-l)(M-l-1)/2-l(l-1)/2}
\sum_{|I|=l}\prod_{a\in I\atop b\not\in I}
\frac{q^{-1}y_a-qy_b}{y_a-y_b}
\prod_{a\not\in I \atop 1\le b\le N}(q^{-1}y_a-z_b) 
\prod_{a\in I\atop 1\le b\le N}(q y_a-z_b)\,,
\nn
\end{align}
where the summation is over subsets $I$ of $\{1,\dots,M\}$.

In particular we have
\begin{align}
&\tilde{I}^c_{N,N}(y,q^{\pm 1}y) =q^{\pm N(N-1)/2}q^{\pm Nc}
\prod_{1\le a,b\le N}(q^{-1}y_a-q y_b)\,.
\label{Icq}
\end{align}

\begin{lem}\label{lem:Iwheel}
Let $N\ge 2$, $c\in\Z$. Then
\begin{align*}
&I^c_{N,N}(y,z)=0\quad \text{if $y_1=q^{-1}z_1, y_2=qz_1$},\\ 
&I^c_{N,N}(y,z)=0\quad \text{if $z_1=q^{-1}y_1, z_2=qy_1$}.
\end{align*} 
\end{lem}
\begin{proof}
The first relation follows from \eqref{Itilde2}. 
Applying the duality \eqref{pr1} we obtain the second. 
\end{proof}

\subsection{Elements $G^{m|n}_{r,N}(s)$}\label{sec:GrN}

In this subsection we introduce rational functions $G^{m|n}_{r,N}(s)$,
$G^{*,m|n}_{r,N}(s)$ and show that they belong to 
the subalgebra $\cB_{m|n}(s)$. Moreover, we will show that $G^{m|n}_{r,N}(s)$,
$G^{*,m|n}_{r,N}(s)$ generate the Bethe algebra
$\cB_{m|n}(s)$, see Theorem \ref{prop:surj_glmn} below.

Conjecturally, generating series of the form 
$\sum_{N=0}^\infty  \gamma_NG^{m|n}_{r,N}(s)u^{-N}$, 
$\sum_{N=0}^\infty  \gamma_N^*G^{*,m|n}_{r,N}(s)u^{-N}$  
with some constants $\gamma_N$, $\gamma^*_N$ correspond to the trace of the universal $R$-matrix of quantum affine algebra $U_q\gl_{m|n}$ computed on the evaluation module with evaluation parameter $u$ obtained from the $r$-th skew-symmetric power of the vector and covector representations, with weights determined by $s$.

We retain the notation in the previous sections, 
assuming $m\ge 1$, $n\ge0$, $m\neq n$.
In particular, we fix 
$s_1,\ldots,s_{m+n}\in (\C^{\times})^{m+n}$ such that  
$\prod_{i=1}^{m+n}s_i=1$.
We shall write
\begin{align*}
&\bfs_i=\prod_{1\le j\le i-1}s_j\,,\quad \bfs_1=1\,. 
\end{align*}
Fixing $N$ we set $\bar{x}_i=\prod_{a=1}^N x_{i,a}$ and
\begin{align*}
&t_i(s)=\bfs_i \frac{\bar{x}_{i-1}}{\bar{x}_i\,,}
\quad
J^c_i(s)=
\bigl(-\bfs_i d^{N}\bigr)^{c}\, \tilde{I}^{c+1}_{N,N}(d x_{i-1},x_i)\,.
\end{align*}
Set further 
\begin{align}\label{Pmn}
P^{m|n}_N
&=
\begin{cases}
\ds{\frac{\prod_{i=1}^{m-1}\prod_{1\le a, b\le N}(x_{i,a}-q_2x_{i,b})
} 
{\prod_{i=1}^{m+n}\prod_{1\le a,b\le N}(x_{i,a}-x_{i+1,b})
}\cdot \bar{x}_m\Delta(x_m)\Delta(x_{m+n})}\,
&\text{if $m,n\ge 1$},
\\
\ds{\prod_{i=1}^{m}\prod_{1\le a, b\le N}\frac{x_{i,a}-q_2x_{i,b}} 
{x_{i,a}-x_{i+1,b}}}
&\text{if $m\ge2$, $n=0$},\\
\ds{
\prod_{1\le a<b\le N}
\frac{(x_{1,a}-q_2x_{1,b})(x_{1,a}-q_2^{-1}x_{1,b})}    
{(x_{1,a}-x_{1,b})^2}}
&\text{if $m=1$, $n=0$}.\\
\end{cases}
\end{align}
\medskip

\begin{dfn}\label{dfn:GrN}
For $r\ge 0$, we define
\begin{align}
G^{m|n}_{r,N}(s)
&=P^{m|n}_N
\sum_{1\le i_1<\cdots<i_p\le m \atop p+c_{m+1}+\cdots+c_{m+n}=r
} t_{i_1}(s)\cdots t_{i_p}(s) \prod_{i=m+1}^{m+n} J_i^{c_i}(s)\,,
\label{dfnG}
\end{align}
The sum is taken over all $i_1,\ldots,i_p$, 
$p\ge0$, and $c_{m+1},\ldots,c_{m+n}\ge 0$, 
subject to the indicated condition.
\qed
\end{dfn}
Terms of $G^{m|n}_{r,N}(s)$ are in one-to-one correspondence with
semi-standard Young tableaux of a column of length $r$ in the alphabet
$\{1,2,\ldots,m+n\}$. Letters are arranged in non-decreasing order, 
where each $i\in\{1,\ldots,m\}$ can appear at most once while 
each $i\in\{m+1,\ldots,m+n\}$ is allowed to repeat $c_i$ times. 
In other words, the terms of $G^{m|n}_{r,N}(s)$ are in one-to-one correspondence with a basis of the $r$-th skew-symmetric power of the vector representation $\C^{m|n}$ of $\gl_{m|n}$. 

Clearly, $G^{m|n}_{0,N}=P^{m|n}_N
\prod_{i=m+1}^{m+n} J_i^{0}(s)$.

When $n=0$, we write $P^m_N=P^{m|0}_N$, $G^m_{r,N}(s)=G^{m|0}_N(s)$. 

\medskip

\begin{example}
If $n=0$, then $G^{m}_{r,N}(s)/P^m_N$ is the 
$r$-th elementary symmetric polynomial
in the variables $t_1(s),\ldots,t_m(s)$.
Hence  $G^{m}_{m,N}(s)=G^m_{0,N}(s)\prod_{i=1}^m\bfs_i$ 
and $G^m_{r,N}(s)=0$ for $r>m$.
\qed
\end{example}

\begin{example}
Let $m=3,n=2$ and $r=2$. Suppressing $s$ we have
\begin{align*}
G^{3|2}_{2,N}/P^{3|2}_N&=t_1t_2J^0_4J^0_5
+t_1t_3 J^0_4J^0_5
+t_2t_3 J^0_4J^0_5
\\
&+t_1J^1_4J^0_5
+t_2J^1_4J^0_5
+t_3J^1_4J^0_5
+t_1J^0_4J^1_5
+t_2J^0_4J^1_5
+t_3J^0_4J^1_5
\\
&+
J^2_4J^0_5
+J^1_4J^1_5
+
J^0_4J^2_5\,.
\end{align*} 
 \qed
\end{example}

\begin{example}
Let $m=2, n=1$ and $N=1$. 
Denote $x_{1,1}=x, x_{2,1}=y,x_{3,1}=z$. 
Then $x$ is bosonic while $y,z$ are fermionic. 
For $N=1$ we have no wheel conditions.

Then one easily checks that $B_{2|1,1}(s)$ 
is three dimensional and consists of elements
$$
\frac{A\, xyz+ B(s_1s_2d^2\,x^2y+s_2d^2\,y^2z+z^2x)  
+C(s_1s_2d^2\,y^2x  +s_1\, x^2z +z^2y )}{(x-y)(y-z)(x-z)},
$$
where $A,B,C$ are arbitrary constants.
We have $\bfs_1=1$, $\bfs_2=s_1,\bfs_3=s_1s_2$, so
$t_1(s)=z/x$, $t_2(s)=s_1x/y$, $t_3(s)=s_1s_2 y/z$.

We have a basis given by 
\begin{align*}
&G^{2|1}_{0,1}= \frac{q^{-1}(1-q_2)^2 xyz}{(x-y)(y-z)(z-x)}
\,,\\
&G^{2|1}_{1,1}= \frac{q^{-1}(1-q_2)^2}{(x-y)(y-z)(z-x)}
\Bigl(s_1s_2d^2 y^2x+s_1x^2z+z^2y-(q_1+q_3^{-1})s_1s_2 xyz\Bigr)\,,
 \\
&G^{2|1}_{2,1}=\frac{1-q_2}{(x-y)(y-z)(z-x)}\Bigl(
-(q-q^{-1})s_1\bigl(s_1s_2d^2x^2y+s_2d^2y^2z+z^2x\bigr)
\\
&+(q^2-q^{-2})s_1s_2d\bigl(s_1s_2d^2y^2x+s_1x^2z+z^2y\bigr)
-(q^3-q^{-3})s_1^2s_2^2d^2 xyz\Bigr)\,. 
\end{align*}
\qed
\end{example}
\medskip

In view of the isomorphism \eqref{cyclmn}, 
we introduce also the following elements for $n\ge1$:
\begin{align}
G^{*,m|n}_{r,N}(s)
&=P^{*,m|n}_N
\sum_{m+1\le i_1<\cdots<i_p\le m+n \atop p+c_{m+1}+\cdots+c_{m+n}=r
} t_{i_1}(s)\cdots t_{i_p}(s) \prod_{i=1}^{m} J_i^{*,c_i}(s)\,,
\label{dfnG*}
\end{align}
where
\begin{align*}
&P^{* m|n}_N=
\frac{\prod_{i=m+1}^{m+n-1}\prod_{1\le a, b\le N}(x_{i,a}-q_2^{-1}x_{i,b})
} 
{\prod_{i=1}^{m+n}\prod_{1\le a,b\le N}(x_{i,a}-x_{i+1,b})
}\cdot \bar{x}_{m+n}\Delta(x_m)\Delta(x_{m+n})\,,
\\
&J^{* c}_i(s)=
\bigl(-\bfs_i d^{-N}\bigr)^{c}\, \tilde{I}^{c+1}_{N,N}(d^{-1} x_{i-1},x_i)\,.
\end{align*}

In terms of the generating series
\begin{align}
G^{m|n}_N(u;s)=\sum_{r=0}^\infty(-u)^{-r}G^{m|n}_{r,N}(s)\,,
\quad
G^{*m|n}_N(u;s)=\sum_{r=0}^\infty(-u)^{-r}G^{*m|n}_{r,N}(s)\,,
\label{genG}
\end{align}
the definition \eqref{dfnG}, \eqref{dfnG*} can be written compactly as
\begin{align}
&
G^{m|n}_N(u;s)=P^{m|n}_N\prod_{i=1}^m (1-t_i(s)/u)
\prod_{i=m+1}^{m+n}J_i(u;s)\,,
\label{genGus}\\
&G^{*,m|n}_N(u;s)=P^{*,m|n}_N\prod_{i=m+1}^{m+n} (1-t_i(s)/u)
\prod_{i=1}^{n}J^*_i(u;s)\,.
\label{genGus*}
\end{align}
Here we set
\begin{align*}
&J_i(u;s)=\sum_{c=0}^\infty(-u)^{-c}J_i^c(s) 
=\sum_{l=0}^N(-1)^l\frac{q^{N-2l}}{1-\bfs_i/(q_2^lq_3^Nu)}
\tilde{I}_{N,N,l}(dx_{i-1},x_i)\,,
\\
&J^*_i(u;s)=\sum_{c=0}^\infty(-u)^{-c}J^{*,c}_i(s) 
=\sum_{l=0}^N(-1)^l\frac{q^{N-2l}}{1-\bfs_i/(q_2^{l}q_1^{-N}u)}
\tilde{I}_{N,N,l}(d^{-1}x_{i-1},x_i)\,,
\end{align*}
with $\tilde{I}_{N,N,l}$ being given in \eqref{Itilde2}. 
\bigskip

Our goal in this subsection is to show the following.
\begin{prop}\label{prop:GinB}
For all $r\ge0$, the elements
$G^{m|n}_{r,N}(s)$, $G^{* m|n}_{r,N}(s)$ 
belong to $\cB_{m|n,N}(s)$. 
\end{prop}

For the proof we prepare two Lemmas.

Let $\bk=(k_1,\ldots,k_{m+n})$ be such that $0\le k_i\le N$ for all $i$.
Consider each summand in $G^{m|n}_{r,N}(s)$ with the substitution
$x_i\to (\xi x_i', x_i'')$, where $x_i'=(x_{i,1},\ldots,x_{i,k_i})$ and 
$x_i''=(x_{i,k_i+1},\ldots,x_{i,N})$: 
\begin{align}
\bigl(P^{m|n}_N  t_{i_1}(s)\cdots t_{i_p}(s) \prod_{i=m+1}^{m+n} 
J_i^{c_i}(s)\bigr)^{\bk}_\xi\,.
\label{PtJ}
\end{align}

Our first task is to estimate its behavior
as $\xi\to\infty$ or $\xi\to 0$.  

\begin{lem}\label{lem:ordPtJ}
(i) The limit of each term \eqref{PtJ} exists
for both $\xi\to\infty$ and $\xi\to 0$. 

(ii) The limit  $\xi\to\infty$ 
is non-zero only if the following are satisfied:
\begin{align*}
&k_{i-1}-k_i \in \{0,1\} \quad \text{if $i\in\{i_1,\ldots,i_p\}$}\,,
\\
&k_{i-1}-k_i \in \{0,-1\} \quad 
\text{if $i\in\{1,\ldots,m\}\backslash\{i_1,\ldots,i_p\}$}\,,
\\
&c_i\ge k_{i-1}-k_i\,\quad
\text{if $i\in\{m+1,\ldots, m+n\}$}.
\end{align*}

(iii) The limit $\xi\to 0$ is non-zero only if 
the following are satisfied:
\begin{align*}
&k_{i-1}-k_i \in \{0,-1\} \quad \text{if $i\in\{i_1,\ldots,i_p\}$}\,,
\\
&
k_{i-1}-k_i \in \{0,1\} \quad 
\text{if $i\in\{1,\ldots,m\}\backslash\{i_1,\ldots,i_p\}$}\,,
\\
&c_i\ge -k_{i-1}+k_i\quad
\text{if $i\in\{m+1,\ldots, m+n\}$}.
\end{align*}
\end{lem}
\begin{proof}
A direct calculation shows that
\begin{align*}
& (P^{m|n}_N)^{\bk}_\xi=
\begin{cases}
O(\xi^{\mu_\infty})& \text{as $\xi\to\infty$},\\
O(\xi^{\mu_0})& \text{as $\xi\to 0$},\\
\end{cases}
\end{align*} 
where
\begin{align*}
\mu_\infty&=
-N\sum_{i=m+1}^{m+n}\bigl(k_{i-1}+k_i\bigr)
-\sum_{i=1}^{m-1}k_i^2+\sum_{i=1}^{m+n} k_{i-1}k_i
-\frac{1}{2}k_m(k_m-1)-\frac{1}{2}k_{m+n}(k_{m+n}+1)
\,,
\\ 
\mu_0&=\sum_{i=1}^{m-1}k_i^2-\sum_{i=1}^{m+n}k_{i-1}k_i
+\frac{1}{2}k_m(k_m+1)+\frac{1}{2}k_{m+n}(k_{m+n}-1)
\,.
\end{align*}
Clearly $\bigl(t_i\bigr)^{\bk}_{\xi}=\xi^{k_{i-1}-k_i}t_i$.
From Lemma \ref{lem:asy-Ic} we have also
\begin{align*}
&\bigl(J_i^{c_i} \bigr)^{\bk}_\xi=
\begin{cases}
 O(\xi^{N(k_{i-1}+k_i)-k_{i-1}k_i}) & \text{as $\xi\to\infty$},
\\
 O(\xi^{k_{i-1}k_i}) & \text{as $\xi\to 0$}.
\end{cases}
\end{align*}

Combining these, we find that 
\begin{align*}
\bigl(P^{m|n}_N  t_{i_1}(s)\cdots t_{i_p}(s) \prod_{i=m+1}^{m+n} 
J_i^{c_i}(s)\bigr)^{\bk}_\xi
= \begin{cases}
O(\xi^{\nu_\infty})& \text{as $\xi\to\infty$},\\
O(\xi^{\nu_0})& \text{as $\xi\to 0$},\\
\end{cases}
\end{align*}
with
\begin{align*}
\nu_\infty
&=\mu_\infty+\sum_{t=1}^p\bigl(k_{i_t-1}-k_{i_t}\bigr)
+\sum_{i=m+1}^{m+n}\bigl(N(k_{i-1}+k_i)-k_{i-1}k_i\bigr)
\\
&=-\frac{1}{2}\sum_{t=1}^p (k_{i_t-1}-k_{i_t})(k_{i_t-1}-k_{i_t}-1)
-\frac{1}{2}\sum_{1\le i\le m\atop i\neq i_1,\ldots,i_p} (k_{i-1}-k_i)(k_{i-1}-k_i+1)
\,,
\\
\nu_0
&=\mu_0+\sum_{t=1}^p\bigl(k_{i_t-1}-k_{i_t}\bigr)
+\sum_{i=m+1}^{m+n}k_{i-1}k_i
\\
&=\frac{1}{2}\sum_{t=1}^p (k_{i_t-1}-k_{i_t})(k_{i_t-1}-k_{i_t}+1)
+\frac{1}{2}\sum_{1\le i\le m \atop
i\neq i_1,\ldots,i_p} (k_{i-1}-k_i)(k_{i-1}-k_i-1)\,.
\end{align*}

For any $k_i$ we have $\nu_\infty\le 0$, which implies that
the limit $\xi\to\infty$ exists.  
The equality holds if and only if $k_{i-1}-k_i\in\{0,1\}$ 
for $i=i_1,\ldots,i_p$ and $k_{i-1}-k_i\in\{0,-1\}$ 
for $1\le i\le m$, $i\neq i_1,\ldots,i_p$. 
Furthermore, by Lemma \ref{lem:Ic-nonzero},
we must have $c_i\ge k_{i-1}-k_{i}$ for $m+1\le i\le m+n$  
in order that the limit be non-zero. 
This proves (ii). 

Assertion (iii) is shown similarly.
\end{proof}

Next we compute the leading coefficients.

\begin{lem}\label{lem:limP}
We have
\begin{align*}
\bigl(P^{m|n}_N\bigr)^{\bk}_\xi&
=\begin{cases}
\xi^{\mu_\infty}\Bigl(X_\infty +O(\xi^{-1})\Bigr)
&\text{as $\xi\to\infty$},
\\
\xi^{\mu_0}\Bigl(X_0 +O(\xi)\Bigr)
&\text{as $\xi\to 0$},
\end{cases}
\end{align*}
where $\mu_\infty,\mu_0$ are as in the proof of Lemma \ref{lem:ordPtJ}
and
\begin{align*}
\frac{X_\infty}{X_0} =(-1)^{k_m(N-k_m)+k_{m+n}(N-k_{m+n})}
\prod_{i=1}^{m}
\Bigl(\frac{\bar{x}_{i-1}}{\bar{x}_i}\Bigr)^{k_i-k_{i-1}}
\prod_{i=m+1}^{m+n}
\frac{\bar{x}_{i-1}^{k_i}\bar{x}_i^{k_{i-1}}}
{\bigl({\bar{x}'_{i-1}}{\bar{x}'_i}\bigr)^N}
\,.
\end{align*} 
\end{lem}
\begin{proof}
This is a direct computation. Assuming $n\ge1$, 
set
\begin{align*}
P'&= \frac{\prod_{i=1}^{m-1}\prod_{1\le a,b\le k_i}(x_{i,a}-q_2x_{i,b})
} 
{\prod_{i=1}^{m+n}\prod_{1\le a\le k_i\atop 1\le b\le k_{i+1}}(x_{i,a}-x_{i+1,b})}
\cdot\bar{x}'_m\,\Delta(x'_m)\Delta(x'_{m+n})\,,
\\
P''&= \frac{\prod_{i=1}^{m-1}\prod_{k_i+1\le a,b\le N}(x_{i,a}-q_2x_{i,b})
} 
{\prod_{i=1}^{m+n}\prod_{k_i+1\le a\le N\atop k_{i+1}+1\le b\le N}(x_{i,a}-x_{i+1,b})}
\cdot \bar{x}''_m\,\Delta(x''_m)\Delta(x''_{m+n})\,.
\end{align*}
Then the factors $X_\infty$, $X_0$ are given respectively as follows.
\begin{align*}
 \frac{X_\infty}{P'P''}&
=(-q_2)^{\sum_{i=1}^{m-1}k_i(N-k_i)}
(-1)^{\sum_{i=1}^{m+n}k_i(N-k_{i-1})}
\\
&\times 
\frac{\prod_{i=1}^{m-1}\bigl(\bar{x}'_i\bigr)^{2(N-k_i)}
}
{\prod_{i=1}^{m+n}
\bigl(\bar{x}'_i\bigr)^{N-k_{i-1}}
\bigl(\bar{x}'_{i-1}\bigr)^{N-k_i}
}
\bigl(\bar{x}'_m\bigr)^{N-k_m}\bigl(\bar{x}'_{m+n}\bigr)^{N-k_{m+n}}
\,,
\\
\frac{X_0}{P'P''}&
=(-q_2)^{\sum_{i=1}^{m-1}k_i(N-k_i)}
(-1)^{\sum_{i=1}^{m+n}k_{i-1}(N-k_i)+k_m(N-k_m)+k_{m+n}(N-k_{m+n})}
\\
&\times 
\frac{\prod_{i=1}^{m-1}\bigl(\bar{x}''_i\bigr)^{2k_i}
}
{\prod_{i=1}^{m+n}
\bigl(\bar{x}''_i\bigr)^{k_{i-1}}
\bigl(\bar{x}''_{i-1}\bigr)^{k_i}
}
\bigl(\bar{x}''_m\bigr)^{k_m}\bigl(\bar{x}''_{m+n}\bigr)^{k_{m+n}}
\,.
\end{align*}
Taking their ratio we obtain the assertion of the Lemma.

In the case $n=0$ the result can be checked easily.
\end{proof}
\medskip

\noindent {\it Proof of Proposition \ref{prop:GinB}.}\quad
It is enough to consider $G^{m|n}_{r,N}(s)$.

First we need to check that the elements 
$G^{m|n}_{r,N}(s)$ 
belong to the shuffle algebra $Sh^0_{m|n,N}$. 

For $1\le i\le m-1$ we have $P^{m|n}_N=0$ if $x_{i,a}=q_2x_{i,b}$.
By Lemma \ref{lem:Iwheel},
for $m+1\le i\le m+n$ we have  $J_{i}^{c_i}(s)=0$ 
if $q_3^{-1}x_{i-1,a}=q_1x_{i-1,b}=x_{i,c}$, 
or if $q_3x_{i,a}=q_1^{-1}x_{i,b}=x_{i-1,c}$.
Therefore $G^{m|n}_{r,N}(s)$ satisfies the wheel conditions.

Next we show the asymptotic conditions. 
 We may assume that 
$k_{i-1}-k_i\in\{-1,0,1\}$ for $1\le i\le m$; 
otherwise $(G^{m|n}_{r,N}(s))^{\bk}_\infty$, 
$(G^{m|n}_{r,N}(s))^{\bk}_0$  are both zero. 
For $\epsilon\in\{-1,0,1\}$ set 
\begin{align*}
&A_{\epsilon}=\{i\mid 1\le i\le m\,,\ k_{i-1}-k_i=\epsilon\}\,,
\end{align*}
so that $ \{1,\ldots,m\}=A_+\sqcup A_0\sqcup A_-$ and
$|A_+|-|A_-|=\sum_{i=1}^m(k_{i-1}-k_i)=k_{m+n}-k_{m}$. 

By Lemma \ref{lem:asy-Ic}, Lemma \ref{lem:limP} and Lemma \ref{lem:ordPtJ},
we obtain
\begin{align}
(G^{m|n}_{r,N}(s))^\bk_\infty
&=Y_\infty\sum_{} t_{i_1}(s)\cdots t_{i_p}(s) 
\prod_{i=m+1}^{m+n}\bigl(-\bfs_id^N\bigr)^{c_i}
\tilde{I}^{c_i+1-k_{i-1}+k_i}_{k_{i-1},k_i}(dx_{i-1}',x_i')
\tilde{I}^{c_{i}+1}_{N-k_{i-1},N-k_i}(d x_{i-1}'',x_i'')\,,
\label{GrNinf}
\\
&Y_\infty=X_\infty\cdot \prod_{i=m+1}^{m+n}(-1)^{(N-k_{i-1})k_i}
d^{k_{i-1}(N-k_i)}
(\bar{x}'_{i-1})^{N-k_i}(\bar{x}'_{i})^{N-k_{i-1}}\,,
\nn
\end{align}
where the sum is taken over $i_1,\ldots,i_p$, $p\ge0$,
$c_{m+1},\ldots,c_{m+n}\ge0$, such that 
\begin{align*}
& A_+\subset\{i_1,\ldots,i_p\}\,,\quad A_-\cap\{i_1,\ldots,i_p\}=\emptyset\,,
\\
&c_i\ge k_{i-1}-k_i\,\quad \text{for $m+1\le i\le m+n$}\,,
\\
&p+\sum_{i=m+1}^{m+n}c_i=r\,.
\end{align*}
Similarly we have
\begin{align}
(G^{m|n}_{r,N}(s))^\bk_{0}
&=Y_0\sum_{} t_{i'_1}(s)\cdots t_{i'_{p'}}(s) \prod_{i=m+1}^{m+n}
\bigl(-\bfs_id^N\bigr)^{c'_i}
\tilde{I}^{c'_i+1}_{k_{i-1},k_i}(dx_{i-1}',x_i')
\tilde{I}^{c'_{i}+1+k_{i-1}-k_i}_{N-k_{i-1},N-k_i}(d x_{i-1}'',x_i'')\,,
\label{GrN0}
\\
&Y_0=X_0\cdot \prod_{i=m+1}^{m+n}(-1)^{k_{i-1}(N-k_i)}
d^{(N-k_{i-1})k_i}
(\bar{x}''_{i-1})^{k_i}(\bar{x}''_{i})^{k_{i-1}}\,,
\nn
\end{align}
where the sum is taken over $i'_1,\ldots,i'_{p'}$, $p'\ge0$,
$c'_{m+1},\ldots,c'_{m+n}\ge0$, such that 
\begin{align*}
& A_-\subset\{i'_1,\ldots,i'_t\}\,,\quad A_+\cap\{i'_1,\ldots,i'_t\}=\emptyset\,,
\\
&c'_i\ge -k_{i-1}+k_i\,\quad \text{for $m+1\le i\le m+n$}\,,
\\
&p'+\sum_{i=m+1}^{m+n}c_i'=r\,.
\end{align*}
 
If we set 
$\{i_1,\ldots,i_p\}=A_+\sqcup\{j_1,\ldots,j_u\}$
and $c_i=c'_i+k_{i-1}-k_i$,
then  \eqref{GrNinf} becomes
\begin{align*}
(G^{mn}_{r,N}(s))^\bk_\infty
&=Y_\infty \prod_{i\in A_+}t_i(s)\cdot
\prod_{i=m+1}^{m+n}\bigl(-\bfs_id^N\bigr)^{k_{i-1}-k_i}
\\
&\times
\sum_{} t_{j_1}(s)\cdots t_{j_u}(s)
\prod_{i=m+1}^{m+n}
\bigl(-\bfs_id^N\bigr)^{c'_i}
\tilde{I}^{c'_i+1}_{k_{i-1},k_i}(dx_{i-1}',x_i')
\tilde{I}^{c'_{i}+1+k_{i-1}-k_i}_{N-k_{i-1},N-k_i}(d x_{i-1}'',x_i'')\,.
\end{align*}
Similarly, by setting $\{i'_1,\ldots,i'_{p'}\}
=A_-\sqcup\{j_1,\ldots,j_u\}$, 
\eqref{GrN0} yields
\begin{align*}
(G^{m|n}_{r,N}(s))^\bk_0 
&=Y_0 \prod_{i\in A_-}t_i(s)
\\
&\times
\sum_{} t_{j_1}(s)\cdots t_{j_u}(s) 
\prod_{i=m+1}^{m+n}
\bigl(-\bfs_id^N\bigr)^{c'_i}
\tilde{I}^{c'_i+1}_{k_{i-1},k_i}(dx_{i-1}',x_i')
\tilde{I}^{c'_{i}+1+k_{i-1}-k_i}_{N-k_{i-1},N-k_i}(d x_{i-1}'',x_i'')\,.
\end{align*}
It is simple to check that the second lines of these equations  
have the same ranges of summation and therefore coincide.
Noting
\begin{align*}
 \frac{\prod_{i\in A_+}t_i(s)}{\prod_{i\in A_-}t_i(s)}
=\prod_{i=1}^m t_i(s)^{k_{i-1}-k_i}
\end{align*}
and using Lemma \ref{lem:limP}, we obtain
\begin{align*}
\frac{(G^{m|n}_{r,N}(s))^\bk_\infty}{(G^{m|n}_{r,N}(s))^\bk_0 }
&=\frac{Y_\infty}{Y_0}\cdot
\frac{\prod_{i\in A_+}t_i(s)}{\prod_{i\in A_-}t_i(s)}
\prod_{i=m+1}^{m+n}\bigl(-\bfs_id^N\bigr)^{k_{i-1}-k_i}
\\
&=\frac{X_\infty}{X_0}\prod_{i=m+1}^{m+n}(-d)^{N(k_{i-1}-k_i)}
\prod_{i=m+1}^{m+n}\frac{(\bar{x}'_{i-1}\bar{x}'_i)^N}
{\bar{x}_{i-1}^{k_i}\bar{x}_{i}^{k_{i-1}}}
\cdot
\prod_{i=1}^m t_i(s)^{k_{i-1}-k_i}
\prod_{i=m+1}^{m+n}\bigl(-\bfs_id^N\bigr)^{k_{i-1}-k_i}
\\
&=\prod_{i=1}^{m+n}s_i^{k_i}\cdot d^{2N(k_m-k_{m+n})}\,.
\end{align*}
This completes the proof of Proposition  \ref{prop:GinB}.
\qed
\bigskip

\subsection{Main theorem}

In this section we prove 
Theorem \ref{prop:surj_glm}
and 
Theorem \ref{prop:surj_glmn} below, which are the main results of this paper. 
We shall make explicit 
the dependence of various objects on parameters $q_1,q_2,q_3$
by writing 
$\cB_{m|n}(q_1,q_2,q_3;s)$, 
$P^{m|n}_{N}(q_1,q_2,q_3)$,  
$G^{m|n}_{r,N}(q_1,q_2,q_3;s)$ 
and so forth.
We recall that 
$\cR_{m|n}=\C[T_{i,\nu}\mid 1\le i\le m+n\,,\ \nu\in \Z_{>0}]$
with $\deg T_{i,\nu}=\nu$. 
\medskip

We start with the special case $m=1,n=0$.
\begin{thm}\cite{FHHSY}\label{thm:FHHSY}
The subalgebra $\cB_1(q_1,q_2,q_3)$  of  $Sh_1(q_1,q_2,q_3)$
is commutative, and is isomorphic to the polynomial algebra $\cR_{1}$.
For each choice of $i\in\{1,2,3\}$, the following
set of elements generate  $\cB_1(q_1,q_2,q_3)$:
\begin{align*}
\epsilon_N(x;q_i)
=\prod_{1\le a<b\le N}\frac{(x_a-q_i x_b)(x_a-q_i^{-1}x_b)} 
{(x_a-x_b)^2}\,,\quad N\ge 0\,.
\end{align*}
\qed
\end{thm}
\medskip

Consider the case $n=0$ with $m$ general. 

\begin{thm}\label{prop:surj_glm}
Let $m\ge1$, and assume that $s$ is generic.

(i) The elements $\{G^m_{r,N}(q_1,q_2,q_3;s)\mid  0\le r\le m-1,\ N\ge 0\}$ 
generate
$\cB_m(q_1,q_2,q_3;s)$.

(ii) The restriction of $\pi_m^{m-1}$ 
to $\cB_{m}(q_1,q_2,q_3;s)$ gives an isomorphism of 
$\Z_{\ge0}$-graded algebras
\begin{align*}
\pi_{m}^{m-1}:\cB_{m}(q_1,q_2,q_3;s)
\longrightarrow \cB_{m-1}(\tilde{q}_1,q_2,\tilde{q}_3;\tilde{s})
\otimes \cB_1(q_1q_3^{-m+1},q_2,q_3^{m})\,,
\end{align*}
where $\tilde{q}_i$ are as in \eqref{tilde-q}, and 
$\tilde{s}=(\tilde{s}_1,\ldots,\tilde{s}_{m-2})$ is given by
\begin{align*}
\tilde{s}_1=(s_2\cdots s_{m-1})^{-1}\,,
\quad \tilde{s}_i=s_i\quad \text{for $2\le i\le m-2$}\,.
\end{align*}

(iii) Algebra $\cB_{m}(q_1,q_2,q_3;s)$ is commutative, and is
isomorphic to the polynomial algebra $\cR_{m}$.
\end{thm}
\begin{proof}
If $m=1$, then  $G^1_{0,N}=\epsilon^{(2)}_N$, 
where $\epsilon^{(2)}_N(x)=\epsilon_N(x;q_2)$. 
Hence the assertions are true by Theorem \ref{thm:FHHSY}. 
Let $m\ge2$, and assume
that (i)--(iii) hold for all $m'<m$. 
For brevity we write $G_{r,N}=G^m_{r,N}(q_1,q_2,q_3;s)$,
$\tilde{G}_{r,N}=G^{m-1}_{r,N}(\tilde{q}_1,q_2,\tilde{q}_3;\tilde{s})$,
$\cB_{m}=\cB_{m}(q_1,q_2,q_3;s)$,
$\cB_{m-1}=\cB_{m-1}(\tilde{q}_1,q_2,\tilde{q}_3;\tilde{s})$,
and 
$\cB_1=\cB_1(q_1q_3^{-m+1},q_2,q_3^{m})$.

Denote by $\cB'_{m}$ the subalgebra of $\cB_m$
generated by the elements $\{G_{r,N}\mid 0\leq r, N\geq 0\}$.
(We recall that $G_{m,N}$ is proportional to $G_{0,N}$ and $G_{r,N}=0$ 
for $r>m$.)  
We use induction on $N$ to prove the following statement:
\begin{align*}
(S_N):\quad &\quad 
\tilde{G}_{r,N}\otimes 1,\ 1\otimes \epsilon^{(2)}_N
\quad \in 
\pi_m^{m-1}\bigl(\cB'_m\bigr)\quad \text{for $r,N\ge0$}\,.
\end{align*}
If $N=0$, then there is nothing to prove.
Assume $(S_{N'})$ for $N'<N$. Then by assertion 
(i) for $m-1$ and  $(S_{N'})$, we have
\begin{align*}
&\cB_{m-1,N'}\otimes 1
=\langle \tilde{G}_{r,K}\otimes 1\mid 0\le r, \ K\le N'\rangle
\subset \pi_m^{m-1}\bigl(\cB'_m\bigr)\,\quad \text{for $N'<N$}.
\end{align*}
Likewise we have
\begin{align*}
&1\otimes \cB_{1,N'}=\langle 1\otimes\epsilon^{(2)}_K\mid \ K\le N'\rangle
\subset \pi_m^{m-1}\bigl(\cB'_m\bigr)\ \quad \text{for $N'<N$}.
\end{align*}
This implies that 
\begin{align*}
\cB_{m-1,L}\otimes\cB_{1,N-L} \subset \pi_m^{m-1}\bigl(\cB'_m\bigr)
\quad \text{for $1\le L\le N-1$}.
\end{align*}

For an element $X\in \cB_{m-1}\otimes \cB_1$, let 
$X_{L,N-L}\in \cB_{m-1,L}\otimes \cB_{1,N-L}$ 
stand for its tensor component. In view of the discussion above, we have
\begin{align*}
\pi_m^{m-1}\bigl(G_{r,N}\bigr)
\equiv \pi_m^{m-1}\bigl(G_{r,N}\bigr)_{N,0}+
\pi_m^{m-1}\bigl(G_{r,N}\bigr)_{0,N}
\quad\bmod  \pi_m^{m-1}\bigl(\cB'_m\bigr)\,.
\end{align*}
On the other hand, a simple computation shows that 
\begin{align*}
\pi_m^{m-1}\bigl(G_N(u)\bigr)_{N,0}
&=C(1-q_1^N/u) 
\cdot \tilde{G}_N(q_1^{N/(m-1)}u)
\otimes 1\,,
\\
\pi_m^{m-1}\bigl(G_N(u)\bigr)_{0,N}
&=C'
(1-q_3^{(m-1)N}/u)\prod_{i=1}^{m-1}(1-\bfs_i q_3^{-N}/u)
\cdot 1\otimes \epsilon_N^{(2)}\,.
\end{align*}
Here $C,C'$ are non-zero constants.
Setting $u=q_1^N$ we find 
$1\otimes \epsilon^{(2)}_N\in \pi_{m}^{m-1}\bigl(\cB'_{m,N}\bigr)$. 
In turn, this implies that
all coefficients of the polynomial $\tilde{G}_N(u)\otimes 1$
belong to $\pi_{m}^{m-1}\bigl(\cB'_{m,N}\bigr)$.
Thus we have proved $(S_N)$.

It follows that the restriction of $\pi_m^{m-1}$ to $\cB'_{m}$
is surjective. In particular we have 
$\dim \cB_{m,N}\ge \dim \cB'_{m,N}\ge \dim \cR_{m,N}$. 
Combining this with the upper estimate in Proposition \ref{prop:Gordon},
we conclude that (i),(ii) hold for $m$. 

Finally assertion (iii) for $m$ 
is a corollary of (ii) together with 
assertion (iii) for $m-1$.

This completes the proof of Theorem.
\end{proof}
\bigskip

\begin{rem}\label{rem:gap}
Statements (i), (iii) are due to \cite{FT} 
except for the following point. 
In the proof of Lemma 2.7, \cite{FT}, it is stated that 
if $F,G\in \cB_m(s)$ then $F*G-G*F$ becomes zero
after clearing the denominator and letting $x_{i,a}\to y_i$.
While true, this is a non-trivial statement that requires a proof. 
In order to circumvent this issue we adopt here a different approach
using fusion homomorphisms.\qed
\end{rem}

\bigskip

Finally we consider the case $m\ge n+1$ and $n\ge1$.

\begin{thm}\label{prop:surj_glmn}
Let $m\ge n+1$, $n \ge1$, and assume that $s$ is generic.

(i) The elements 
$\{G^{m|n}_{r,N}(q_1,q_2,q_3;s),
G^{*,m|n}_{r,N}(q_1,q_2,q_3;s)\mid 0\le r,\ N\ge 0\}$ 
generate
$\cB_{m|n}(q_1,q_2,q_3;s)$.

(ii) The restriction of $\pi_{m|n}^{m|n-1}$ 
to $\cB_{m|n}(q_1,q_2,q_3;s)$ gives an isomorphism
\begin{align*}
\pi_{m|n}^{m|n-1}:\cB_{m|n}(q_1,q_2,q_3;s)
\longrightarrow \cB_{m|n-1}(\bar{q}_1,q_2,\bar{q}_3;{s})
\otimes \cB_1(q_2q_3^{-m+n},q_2^{-1},q_3^{m-n})\,,
\end{align*}
where $\bar{q}_i$ are as in \eqref{qbar}. 

(iii) Algebra $\cB_{m|n}(q_1,q_2,q_3;s)$ is commutative, and 
is isomorphic to the polynomial algebra $\cR_{m|n}$.
\end{thm}
\begin{proof}
The proof is quite similar to that of Theorem \ref{prop:surj_glm}.
We use the latter as a base of induction
and the following formulas:
\begin{align*}
&\pi_{m|n}^{m|n-1}\bigl(G_N(u)\bigr)_{N,0}
=C_1\cdot
\frac{1}{1-\bfs_{m+n}{q}_1^N/u}\cdot
\bar{G}_N(q_1^{-N/(m-n+1)} 
u) \otimes 1\,,
\\
&\pi_{m|n}^{m|n-1}\bigl(G_N(u)\bigr)_{0,N}
=C_2\frac{\prod_{i=1}^m(1-\bfs_i/(q_3^Nu))}
{\prod_{i=m+1}^{m+n-1}(1-\bfs_i/(q_3^Nu))}
\sum_{l=0}^N\frac{C_{N,l}}{1-\bfs_{m+n}/(q_2^lq_3^Nu)}
1\otimes \bigl(\epsilon^{(1)}_{N-l}*\epsilon^{(3)}_{l}\bigr)\,,
\\
&\pi_{m|n}^{m|n-1}\bigl(G^*_N(u)\bigr)_{N,0}
=C_3\ (1-\bfs_{m+n}q_1^{-N}/u)\cdot \bar{G}^*_N(
q_1^{N/(m-n+1)}  
u)\otimes 1\,,
\\
&\pi_{m|n}^{m|n-1}\bigl(G^*_N(u)\bigr)_{0,N}
=C_4\ \frac{\prod_{i=m+1}^{m+n-1}(1-\bfs_i q_3^N/u)
\cdot (1-\bfs_{m+n}q_3^{(m-n+1)N}/u)}
{\prod_{i=1}^m(1-\bfs_i q_3^N/u)}\cdot 1\otimes \epsilon^{(2)}_N\,.
\end{align*}
Here 
$\epsilon^{(1)}_N(z)=\epsilon_N(z;q_2q_3^{-m+n})$,
$\epsilon^{(3)}_N(z)=\epsilon_N(z;q_3^{m-n})$,  and 
$C_1,\ldots,C_4$, $C_{N,l}$ are non-zero constants.

\end{proof}

\appendix

\section{Trace construction}

In this Appendix we discuss the relation between 
the subalgebra $\cB_{m|n}(s)$ of $Sh_{m|n}$
and the commutative subalgebra
of the quantum toroidal algebra $\E_{m|n}$ of type $\mathfrak{gl}_{m|n}$.

Algebra $\E_{m|n}$ has generators $E_{i,k},F_{i,k}$, $K_{i,0}^{\pm1}$, $H_{i,r}$, where
$i\in \Z/(m+n)\Z$, $k\in\Z$, $r\in \Z\backslash\{0\}$, and an invertible central element $C$. 
The defining relations are presented in terms of the generating series
$E_i(z)=\sum_k E_{i,k}z^{-k}$, $F_i(z)=\sum_k F_{i,k}z^{-k}$, 
$K^{\pm}_i(z)=K_{i,0}^{\pm1}\exp\bigl(\pm (q-q^{-1})\sum_{\pm r>0}H_{i,r}z^{-r}\bigr)$. 
For the details we refer the reader to \cite{BM1}. 
We note only that the quadratic relations for the currents
$E_i(z),F_i(z)$ can be written as 
\begin{align*}
&\omega_{i,j}(z,w)E_i(z)E_j(w)=(-1)^{|i||j|}\omega_{j,i}(w,z)E_j(w)E_i(z)\,, 
\\
&(-1)^{|i||j|}\omega_{j,i}(w,z)F_i(z)F_j(w)
=\omega_{i,j}(z,w)F_j(w)F_i(z)\,, 
\end{align*}
where $\omega_{i,j}(z,w)$ are the structure functions of $Sh_{m|n}$.
We use also the Hopf algebra structure of $\E_{m|n}$ given by the Drinfeld coproduct, see 
\cite{BM1}, Section 2.4.

\subsection{The case $\E_m$}\label{ap1 sec}

First we restrict to  $\E_m=\E_{m|0}$, 
and collect known facts following \cite{FT}. 

Let $\E^+_m$ (resp. $\E^-_m$) be the subalgebra of $\E_m$
generated by $\{E_{i,k}\}$ (resp. $\{F_{i,k}\}$).
There is an anti-isomorphism of algebras $\E^+_m\to \E^-_m$ sending 
$E_{i,k}$ to $F_{i,k}$. 
These subalgebras have realization as shuffle algebras due to \cite{N2}. 
Namely we have an isomorphism of algebras 
\begin{align*}
\Psi^\pm: \E_m^\pm \longrightarrow Sh^\pm_m
\end{align*}
sending $E_{i,k}$ (resp. $F_{i,k}$) to the element $x_i^k\in Sh_{m,{\bf 1}_i}$, 
where $Sh_m^-=Sh_m$, and $Sh_m^+=(Sh_m)^{op}$ is the 
opposite algebra 
defined by the structure functions $\omega^{op}_{i,j}(z,w)=\omega_{j,i}(w,z)$.
Note that the zeroes of $\omega^{op}_{i,j}(z,w)$ 
are obtained from those of $\omega_{i,j}(z,w)$ by the interchange 
$q_1\leftrightarrow q_3^{-1}$, which leaves the wheel conditions unchanged.  

Let $\cB_m^+$ (resp. $\cB_m^-$) be the Hopf subalgebra of $\E_m$
generated by $E_{i,k}, K_{i,0}^{\pm1}, H_{i,-r}$ and $C$ 
(resp. $F_{i,k}$, $K_{i,0}^{\pm1}$, $H_{i,r}$ and $C$), 
$i\in\Z/m\Z$, $k\in\Z$, $r>0$. There is a non-degenerate Hopf pairing
\begin{align*}
\langle~, ~\rangle : & \cB^+\times \cB^- \longrightarrow \C
\end{align*}
given in \cite{FT}, Theorem 1.7.
The isomorophism $\Psi^+$ can be described explicitly in terms of this pairing.
To this aim, introduce the notation 
\begin{align*}
F_{i_1,\ldots,i_k}(z_1,\ldots,z_k)
=\prod_{1\le a<b\le k}\omega_{i_b,i_a}(z_b,z_a)\cdot
F_{i_1}(z_1)\cdots F_{i_k}(z_k)\,. 
\end{align*}
The quadratic relations for $F_i(z)$ imply the symmetry
\begin{align*}
F_{i_{\sigma(1)},\ldots,i_{\sigma(k)}}(z_{\sigma(1)},\ldots,z_{\sigma(k)}) 
=F_{i_1,\ldots,i_k}(z_1,\ldots,z_k)\quad 
\text{for $\sigma\in\mathfrak{S}_k$}\,.
\end{align*}
We have then (see \cite{FT}, Lemma 3.9)
\begin{align}
(q-q^{-1})^N
\Psi^+(X)(x_1,\ldots,x_m)
=\langle X, F_{1^{N_1},\ldots,m^{N_m}}(x_1,\ldots,x_m)\rangle 
\quad \text{for $X\in\E^+_{m,N_1,\ldots,N_m}$}\,, 
\label{repro}
\end{align}
where $N=\sum_{i=1}^mN_i$, 
$i^{N_i}=i,\ldots,i$ ($N_i$ times) and
$x_i=(x_{i,1},\ldots,x_{i,N_i})$, $1\le i\le m$. 

A standard way to construct commutative subalgebras in $\E_m$ is to 
consider transfer matrices, namely appropriate traces of the universal 
$R$-matrix $\cR$. 
More specifically, consider the 
Fock modules $\F_r$, $r=0,1,\ldots,m$, of $\E_m$. 
These modules are graded, $\F_r=\bigoplus_{n=-\infty}^0 \F_{r,n}$,
the 
top degree subspace $\F_{r,0}$ being the $r$-th fundamental representation of 
$U_q\mathfrak{gl}_m$. Fixing parameters 
$u_i\in\C^{\times}$,   
$1\le i\le m-1$, 
consider the weighted trace 
\begin{align}
(\id\otimes \phi_r)\cR\,, 
\quad 
\phi_{r}(Y)=\Tr_{\F_{r}}
\Bigl(P_0\prod_{i=1}^{m-1}
u_i^{-\bar{\Lambda}_i}\  
Y\Bigr).
\label{TrR}
\end{align}
Here $P_0:\F_r\to \F_{r,0}$ stands for the projection, 
and $\bar{\Lambda}_i$ are the fundamental weights of $\mathfrak{gl}_m$. 
The universal $R$-matrix
$\cR$ is an element of the completed tensor product
$\cB^+_m\otimes\cB^-_m$. 
Due to the Yang-Baxter equation, the elements $(\id\otimes \phi_r)\cR$, 
$r=0,\ldots,m$, 
mutually commute.

The functional $\phi_r:\cB^-\to \C$ annihilates 
$H_{i,r}$, $r>0$. Moreover the dependence 
of $(\id\otimes \phi_r)\cR$ 
on $K_{i,0}$ is such that it
is a linear combination of elements
of the form $X f(s_1,\ldots,s_{m})$, where 
$X\in\E^+_m$ has weight $0$ and 
$f$ is a function in $s_i=u_i K^{-1}_{i,0}$.
Since $s_i$'s commute with $X$,
 \eqref{TrR} is essentially an element of $\E^+_m$
in which $s_i$'s enter as parameters.

On the other hand, since $\cR$ is the canonical element we have
\begin{align}
\langle (\id\otimes \phi_r)\cR, Y \rangle =\phi_r(Y)\quad 
\text{for $Y \in \cB^-_m$}. 
\label{canR}
\end{align}
Using \eqref{canR} and \eqref{repro} 
one can rewrite \eqref{TrR} 
as commutative elements in the shuffle algebra $Sh_m$.

Thus we conclude:
\begin{prop}
The collection of elements
\begin{align*}
\prod_{i=1}^m\prod_{1\le a<b\le N}\omega_{i,i}(x_{i,b},x_{i,a})\prod_{1\le i<j\le m}\prod_{1\le a,b\le N}
\omega_{j,i}(x_{j,b},x_{i,a})\cdot 
\phi_r\Bigl(\prod_{1\le i\le m}^{\curvearrowright}\prod_{1\le a\le N}^{\curvearrowright}F_i(x_{i,a})
\Bigr), 
\end{align*}
where $0\le r\le m$, $N\ge0$, constitutes a commutative family in $Sh_m$.
\end{prop}
Using the explicit realization of $F_i(z)$ by vertex operators, 
one checks that these elements coincide with $G^m_{r,N}(s)$ up to a scalar multiple, with the identification $s_i=u_i K^{-1}_{i,0}$.

\subsection{The case $\E_{2|1}$}

In the previous section, we explained the 
scheme for constructing commutative elements in $Sh_m$. 
We expect that the same method carries over to the case $Sh_{m|n}$ with $m\neq n$. 
In place of $\F_r$ we use admissible modules $\F_\Lambda$ 
which have a realization in terms of free fields \cite{BM1}. A new feature is that,
for $i$ beyond the equator , the currents $E_i(z),F_i(z)$ are no longer a single 
vertex operator but a sum of them. 

As an example, consider the case $m=2,n=1$. 
(In the following we quote relevant formulas from \cite{BM1}.
For the unexplained notation we refer the reader to Section 3 thereof.)
The $F_i$ currents have the form
\begin{align*}
F_1(z)=\Gamma_1^-(z),\quad F_2(z)=:\Gamma_2^-(z)\partial_z\bigl[C^-_2(d^2z)\bigr]:,
\quad
F_0(z)=d:\Gamma^-_0(z)C^+_2(dz):\,,
\end{align*}
where $\Gamma^-_i(z)$, $C^\pm_i(z)$ are certain vertex operators, and 
$\partial_z$ signifies the $q$-difference operator
\begin{align*}
\partial_zf(z)=\frac{f(qz)-f(q^{-1}z)}{(q-q^{-1})z}\,. 
\end{align*}
Notice that $F_2(z)$ is a sum of two vertex operators.

In the simplest case $\Lambda=0$, 
the top degree subspace $\F_{\Lambda,0}$ is one dimensional with basis $v_0=1$.
The weighted trace is simply the action of $\prod_{1\le i\le 3}^{\curvearrowright}\prod_{1\le a\le N}^{\curvearrowright}F_i(x_{i,a})$.
We find 
\begin{align*}
P^{2|1}_N \tilde{I}^1_{N,N}(dx_2,x_3)=G^{2|1}_{0,N}\,. 
\end{align*}
The appearance of the functions $I^c_{M,N}$ in section \ref{sec:IcN}
is due to the presence of $\partial_z$ in $F_2(z)$.
 
For $\Lambda=\bar{\Lambda}_1$, $\F_{\bar{\Lambda}_1,0}$ is three dimensional
with basis $v_1=e^{\bar{\Lambda}_1}$, $v_2=e^{-\bar{\alpha}_1}e^{\bar{\Lambda}_1}$,
and $v_3=c_{2,-1}e^{-\bar{\alpha}_1-\bar{\alpha}_2-c_2}e^{\bar{\Lambda}_1}$.
Taking trace we obtain
\begin{align*}
P^{2|1}_N \Bigl(\frac{\bar{x}_3}{\bar{x}_1}\tilde{I}^1_{N,N}(dx_2,x_3)
+s_1\frac{\bar{x}_1}{\bar{x}_2}\tilde{I}^1_{N,N}(dx_2,x_3)
-s_1s_2 d^N \frac{\bar{x}_2}{\bar{x}_3}\tilde{I}^2_{N,N}(dx_2,x_3)
\Bigr)
=G^{2|1}_{1,N}(s)\,.
\end{align*}
Here, computing the action on $v_3$ we use further the identity
\begin{align}\label{identity}
\frac{1}{\Delta(y)}\prod_{a=1}^N\frac{q T_{y_a,q}-q^{-1}T_{q,y_a}^{-1}}{q-q^{-1}}
\Bigl[\frac{\Delta(y)}{\Pi(y,z)}\Bigl(1-\sum_{a,b}(y_a-z_a)(y_b^{-1}-z_b^{-1})\Bigr)\Bigr]
=\frac{\bar{y}}{\bar{z}}I^2_{N,N}(y,z)\,.
\end{align}

Further examples suggest 
the formula for $G^{m|n}_{r,N}(s)$ in Section \ref{sec:GrN}. 
We expect that these functions coincide with 
weighted  traces of products of $F_i(z)$ given by free fields. 
To verify this in general, we would need a host of identities 
of the type given above. 
For that reason we have chosen to prove the commutativity of 
$G^{m|n}_{r,N}(s)$ directly. 
\bigskip

\bigskip

{\bf Acknowledgments.\ }
MJ is partially supported by 
JSPS KAKENHI Grant Number JP19K03549. 
EM is partially supported by Simons Foundation grant number \#709444.

BF thanks Jerusalem University for hospitality. EM thanks Rikkyo University, where a part of this work was done, for hospitality.

\end{document}